\documentclass[11pt]{article}
\usepackage{enumerate}
\usepackage{amssymb,a4wide,latexsym,makeidx,epsfig,fleqn}
\usepackage{amsthm}
\usepackage{amsmath}
\usepackage{color}
\newtheorem{theorem}{Theorem}[section]
\newtheorem{lemma}[theorem]{Lemma}
\newtheorem{corollary}[theorem]{Corollary}
\newtheorem{remark}[theorem]{Remark}

\newtheorem{conjecture}[theorem]{Conjecture}

\begin{document}
\textwidth 150mm \textheight 225mm
\title{On the $A_\alpha$ spectral radius of strongly connected digraphs \thanks{Supported by the National Natural Science Foundation of China (Nos. 11871398 and 12001434).}}
\author{{Weige Xi$^{a,b}$, Ligong Wang$^{b,c}$\footnote{Corresponding author.}}\\
{\small $^{a}$ College of Science, Northwest A\&F University, Yangling, Shaanxi 712100, China.}\\
{\small $^{b}$ School of Mathematics and Statistics, Northwestern Polytechnical University,}\\{\small  Xi'an, Shaanxi 710129, P.R. China.}\\
 {\small $^{c}$ Xi'an-Budapest Joint Research Center for Combinatorics,}
 \\{\small   Northwestern Polytechnical University, Xi'an, Shaanxi 710129, P.R. China.}
 \\{\small E-mail: xiyanxwg@163.com, lgwangmath@163.com}}
\date{}
\maketitle
\begin{center}
\begin{minipage}{135mm}
\vskip 0.3cm
\begin{center}
{\small {\bf Abstract}}
\end{center}
{\small Let $G$ be a digraph with adjacency matrix $A(G)$. Let
$D(G)$ be the diagonal matrix with outdegrees of
vertices of $G$. Nikiforov \cite{Niki} proposed to study the convex combinations
of the adjacency matrix and diagonal matrix of the degrees of undirected graphs.
Liu et al. \cite{LWCL} extended the definition to digraphs. For any real $\alpha\in[0,1]$, the matrix $A_\alpha(G)$ of
a digraph $G$ is defined as
$$A_\alpha(G)=\alpha D(G)+(1-\alpha)A(G).$$
The largest modulus of the eigenvalues
of $A_\alpha(G)$ is called the $A_\alpha$ spectral radius of $G$, denoted by
$\lambda_\alpha(G)$. This paper proves some extremal results about the spectral radius $\lambda_\alpha(G)$
that generalize previous results about $\lambda_0(G)$ and $\lambda_{\frac{1}{2}}(G)$.

In particular, we characterize the extremal
digraph with the maximum (or minimum) $A_\alpha$ spectral radius
among all $\widetilde{\infty}$-digraphs and $\widetilde{\theta}$-digraphs on $n$ vertices.
Furthermore, we determine the digraphs with the second and the third
minimum $A_\alpha$ spectral radius among all strongly connected bicyclic digraphs.
For $0\leq\alpha\leq\frac{1}{2}$, we also
determine the digraphs with the second, the third and the fourth
minimum $A_\alpha$ spectral radius among all strongly connected digraphs on $n$ vertices. Finally, we characterize the digraph with
the minimum $A_\alpha$ spectral radius among all strongly connected bipartite digraphs which contain a complete bipartite subdigraph.
 \vskip 0.1in \noindent {\bf Key Words}: \ Strongly connected digraph, Signless Laplacian, Adjacency matrix, $A_\alpha$ spectral radius. \vskip
0.1in \noindent {\bf AMS Subject Classification (2010)}: \ 05C50,15A18 }
\end{minipage}
\end{center}

\section{Introduction }

Let $G=(V(G),E(G))$
be a digraph with vertex set
$V(G)=\{v_1,v_2,\ldots,v_n\}$ and arc set
$E(G)$. If there is an arc from $v_i$ to $v_j$, we indicate this by writing
$(v_i,v_j)$, call $v_j$ the head of $(v_i,v_j)$, and $v_i$ the tail
of $(v_i,v_j)$, respectively. A digraph $G$ is
called strongly connected if for every pair of vertices $v_i,v_j\in
V(G)$, there exists a directed path from $v_i$ to
$v_j$ and a directed path from $v_j$ to $v_i$. For any vertex $v_i$,
let $N_i^{+}=\{v_j\in V(G) \mid (v_i,v_j)\in
E(G)\}$ denote the out-neighbors of $v_i$. Let $d_i^{+}=|N_i^{+}|$
denote the outdegree of the vertex $v_i$ in the digraph $G$. Let $P_{n}$ and
$C_{n}$ denote the directed path and the
directed cycle on $n$ vertices, respectively. Let $\overset{\longleftrightarrow}{K_{n}}$
denote the complete digraph on $n$ vertices in which for any two distinct vertices
$v_i,v_j\in V(\overset{\longleftrightarrow}{K_{n}})$, there are arcs
$(v_i,v_j)$ and $(v_j,v_i)\in E(\overset{\longleftrightarrow}{K_{n}})$. Suppose
$P_k=v_1v_2\dots v_k$, we call $v_1$ the
initial vertex of the directed path
$P_k$, and $v_k$ the terminal vertex of the
directed path $P_k$.  All digraphs considered in this paper are
simple digraphs, i.e., without loops and multiple arcs.

Let $G=(V(G),E(G))$ be a digraph, if $V(G)=U\cup W$, $U\cap W=\emptyset$
and for any arc $(v_i,v_j)\in E(G)$, $v_i\in U$ and $v_j\in W$ or $v_i\in W$ and $v_j\in U$,
then the digraph $G$ is called a bipartite digraph. Let
$\overleftrightarrow{K_{p,q}}$ be a complete bipartite digraph
obtained from a complete bipartite undirected graph $K_{p,q}$ by replacing each edge with a
pair of oppositely directed arcs.

The $\infty$-digraph \cite{LS} is a digraph on $n$ vertices obtained from two
directed cycles $C_{k+1}$ and
$C_{l+1}$ by identifying a vertex of
$C_{k+1}$ with a vertex of
$C_{l+1}$, denoted by $\infty(k,l)$, $1\leq k\leq
l$ and $k+l+1=n$ (see Figure \ref{fig:c} when $s=2$).
The $\theta$-digraph consists of three
directed paths $P_{a+2}$, $P_{b+2}$, and $P_{c+2}$ such that the
initial vertex of $P_{a+2}$ and $P_{b+2}$ is the terminal vertex of
$P_{c+2}$, and the initial vertex of $P_{c+2}$ is the terminal vertex of
$P_{a+2}$ and $P_{b+2}$, denoted by $\theta(a,b,c)$, where $a\leq b$ and $a+b+c+2=n$ (see Figure \ref{fig:d} when $s=2$ ).

A digraph $G$ is called a strongly connected bicyclic
digraph if $G$ is strongly connected and
$|E(G)|=|V(G)|+1$. Note that each strongly connected bicyclic digraph is either a $\theta$-digraph or a $\infty$-digraph.

For a digraph $G$, let
$A(G)=(a_{ij})_{n\times n}$ be the adjacency matrix
of $G$, where $a_{ij}=1$ whenever $(v_i,v_j)\in
E(G)$, and $a_{ij}=0$ otherwise. Let
$D(G)$ be the diagonal matrix with outdegrees of
vertices of $G$. The sum of $A(G)$ and $D(G)$ is called the signless Laplacian matrix $Q(G)$, which has been
extensively studied since then. More detailed information about this research see \cite{HoYo,LWZ,XiWa1,XiWa2}, and their references.
Nikiforov \cite{Niki} proposed to study the convex linear combinations of the adjacency matrix and diagonal matrix of degrees of undirected graphs,
which give a unified theory of adjacency spectral and signless Laplacian spectral theories. Liu et al. \cite{LWCL} extended the definition to digraphs, they
proposed to study the convex combinations $A_\alpha(G)$ of $A(G)$ and $D(G)$ of the digraph $G$, which is defined as
$$A_\alpha(G)=\alpha D(G)+(1-\alpha)A(G), \ \ 0\leq \alpha\leq1.$$
Obviously,
$$A(G)=A_0(G), \ \ \  D(G)=A_1(G), \ \ \ \textrm{and} \ \ \ Q(G)=2A_{\frac{1}{2}}(G).$$
Since $A_{\frac{1}{2}}(G)$ is essentially equivalent to $Q(G)$, in this paper
we take $A_{\frac{1}{2}}(G)$ as an exact substitute for $Q(G)$. The spectral radius of
$A_\alpha(G)$, i.e., the largest modulus of the eigenvalues
of $A_\alpha(G)$, is called the $A_\alpha$ spectral radius of $G$, denoted by
$\lambda_\alpha(G)$. The $A_\alpha$ spectral radius of undirected graphs has been studied in the literature, see \cite{LiLi,LL,NR,NO,XLLS}.
Recently, Liu et al. \cite{LWCL} determined the unique digraph which attains the maximum (resp. minimum) $A_\alpha$ spectral radius among all strongly connected bicyclic digraphs.
Xi et al. \cite{XiWa4} characterized the digraphs which attain the maximum or minimum $A_\alpha$ spectral radius among all strongly connected digraphs with given girth, clique number, vertex connectivity and arc connectivity, respectively. Ganie and Baghipur \cite{GM} obtained some lower and upper bounds on the $A_\alpha$ spectral radius of digraphs and characterized the extremal digraphs attaining these bounds. We are interested in the $A_\alpha$ spectral radius of some other strongly connected digraphs.

If $\alpha=1$, $A_1(G)=D(G)$ the diagonal matrix with outdegrees of
vertices of the digraph $G$ which is not interesting. So we only consider the cases $0\leq \alpha<1$ in the rest of this paper.
If $G$ is a
strongly connected digraph, then it follows from the Perron Frobenius Theorem \cite{HJ} that $\lambda_\alpha(G)$ is an eigenvalue of $A_\alpha(G)$,
and there is a unique positive unit eigenvector corresponding to $\lambda_\alpha(G)$. The positive unit eigenvector
corresponding to $\lambda_\alpha(G)$ is called the Perron vector of $A_\alpha(G)$.

Spectral graph theory is a fast growing branch of algebraic graph theory. The most studied problems are those of characterization of extremal graphs, such as determine
the maximum or minimum spectral (signless Laplacian spectral) radius over various families of graphs. Recently,
in \cite{LSWY}, Lin et al. determined
the digraphs with the minimum $A_0$ spectral radius among all strongly connected
digraphs with given clique number and girth. In \cite{LD}, Lin and Drury gave the extremal digraphs with the maximum $A_0$ radius among all strongly connected
digraphs with given arc connectivity. In \cite{LS1}, Lin and Shu characterized the digraph which has the maximum $A_0$ spectral radius among all strongly connected
digraphs with given dichromatic number. In \cite{HoYo}, Hong and You determined the digraph which achieves the minimum (or maximum)
$A_{\frac{1}{2}}$ spectral radius among all strongly connected digraphs with some given parameters such as clique number, girth or vertex connectivity.
In \cite{XiWa2}, Xi and Wang determined the extremal digraph with the maximum $A_{\frac{1}{2}}$ spectral
radius among all strongly connected digraphs with given dichromatic number. The main goal of this paper is to extend some results
on maximum or minimum $A_0$ spectral radius and $A_{\frac{1}{2}}$ spectral radius for all $\alpha\in[0,1)$.

The rest of the paper is structured as follows. In Section 2, we will determine the extremal digraphs which achieve
the maximum and minimum $A_\alpha$ spectral radius among all $\widetilde{\infty}$-digraphs and $\widetilde{\theta}$-digraphs
(their definitions can be found in Section 2).
In Section 3, for $0\leq\alpha\leq\frac{1}{2}$, we determine the digraphs which achieve the second, the third and the forth minimum $A_\alpha$ spectral radius
of strongly connected digraphs on $n$ vertices. For general case, we propose a conjecture. In Section 4, we determine the extremal digraph which attains
the minimum $A_\alpha$ spectral radius of strongly
connected bipartite digraphs which contain a complete bipartite subdigraph. The results in our paper generalize some results in \cite{CCL,GL,LZ,LWZ,LWCL}.

\section{The $A_\alpha$ spectral radius of $\widetilde{\infty}$-digraphs and $\widetilde{\theta}$-digraphs}

We have known the $\theta$-digraphs and $\infty$-digraphs. The generalized strongly connected
$\widetilde{\infty}$-digraph is a digraph consisting of $s$ ($s\geq2$) directed cycles with just a vertex in common (as shown in Figure \ref{fig:c}),
denoted by $\widetilde{\infty}(k_1,k_2,\ldots,k_s)$ such that $\sum\limits_{i=1}^sk_i+1=n$. Without loss of
generality, let $1\leq k_i\leq k_{i+1}$ for $i=1,2,\ldots, s-1$. The generalized strongly connected $\widetilde{\theta}$-digraph
consists of $s+1$ ($s\geq2$) directed paths  $P_{k_1+2},\ldots,P_{k_s+2}$ and $P_{l_1+2}$ such that the initial
vertex of $P_{k_1+2},\ldots,P_{k_s+2}$ is the terminal vertex of $P_{l_1+2}$, and the initial
vertex of $P_{l_1+2}$ is the terminal vertex of $P_{k_1+2},\ldots,P_{k_s+2}$ (as shown in Figure \ref{fig:d}), denoted by $\widetilde{\theta}(k_1,k_2,\ldots,k_s,l_1)$
such that $\sum\limits_{i=1}^sk_i+l_1+2=n$.  Without loss of
generality, let $0\leq k_i\leq k_{i+1}$ for $i=1,2,\ldots, s-1$. Note that any $\widetilde{\theta}(k_1,k_2,\ldots,k_s,l_1)$-digraph
 contains $s$ directed cycles.

\begin{figure}[htbp]
\begin{minipage}[t]{0.5\linewidth}
\centering
\includegraphics[scale=0.6]{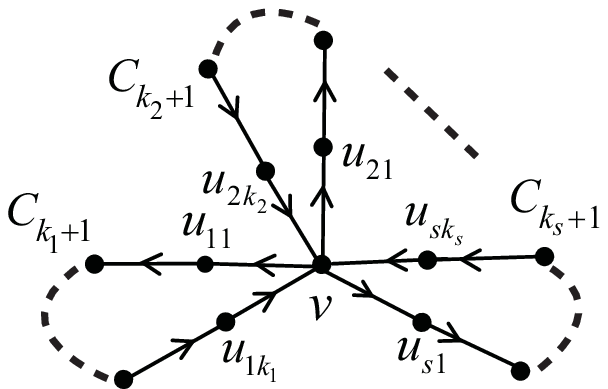}
\caption{ The digraph $\widetilde{\infty}(k_1,k_2,\ldots,k_s)$}\label{fig:c}
\end{minipage}
\begin{minipage}[t]{0.45\linewidth}
\centering
\includegraphics[scale=0.6]{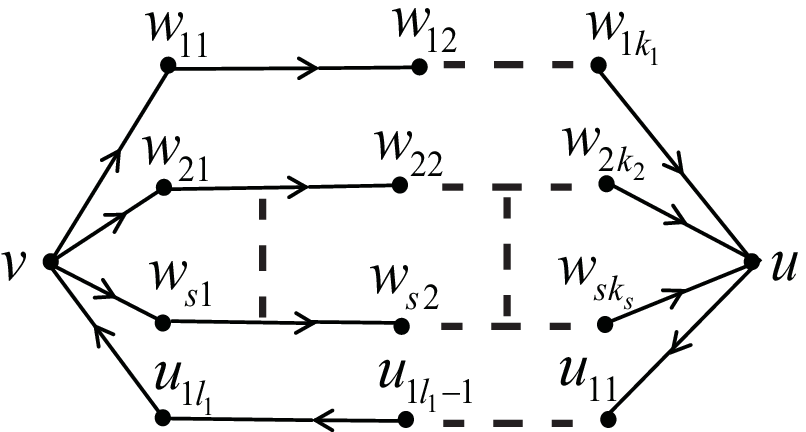}
\caption{ The digraph $\widetilde{\theta}(k_1,k_2,\ldots,k_s,l_1)$}\label{fig:d}
\end{minipage}
\end{figure}

Guo and Liu \cite{GL} characterized the digraph which attains the minimum and maximum $A_0$ spectral radius among all $\widetilde{\theta}$-digraphs and $\widetilde{\infty}$-digraphs on $n$ vertices, respectively. Li et al. \cite{LWZ} determined that
the digraph which attains the minimum and maximum $A_{\frac{1}{2}}$ spectral radius among all $\widetilde{\theta}$-digraphs
and $\widetilde{\infty}$-digraphs on $n$ vertices, respectively. We generalize their results to $0\leq\alpha<1$. Moreover, Li and Zhou \cite{LZ} characterized digraphs which
achieve the second and the third minimum $A_{\frac{1}{2}}$ spectral radius among all strongly connected bipartite digraphs. We also generalize their results to $0\leq\alpha<1$.

\noindent\begin{lemma}\label{le:1} (\cite{HJ}) Let $M$ be an $n\times n$
nonnegative irreducible matrix with spectral radius
$\varrho(M)$ and row sums $s_1,s_2,\ldots,s_n$. Then $$\min_{1 \leq
i \leq n}s_i\leq\varrho(M)\leq\max_{1 \leq i\leq n}s_i.$$ Moreover,
one of the equalities holds if and only if the row sums of $M$ are
all equal.
\end{lemma}

\noindent\begin{lemma}\label{le:14} For any $p, q\in\{1, 2,\ldots,s\}$, if $2\leq k_p\leq k_q$, then we have
\begin{align*}
&\lambda_\alpha(\widetilde{\infty}(k_1,k_2,\ldots,k_{p-1},k_p-1,k_{p+1},\ldots,k_{q-1},k_q+1,k_{q+1},\ldots,k_s))\\
&>\lambda_\alpha(\widetilde{\infty}(k_1,k_2,\ldots,k_{p-1},k_p,k_{p+1},\ldots,k_{q-1},k_q,k_{q+1},\ldots,k_s)).
\end{align*}
\end{lemma}

\begin{proof} Let $G=\widetilde{\infty}(k_1,k_2,\ldots,k_{p-1},k_p,k_{p+1},\ldots,k_{q-1},k_q,k_{q+1},\ldots,k_s)$ be
a digraph shown in Figure \ref{fig:c}.
Suppose $X=(x_v,x_{1,1},x_{1,2},\ldots,x_{1,k_1},x_{2,1},x_{2,2},\ldots,x_{2,k_2},\ldots, x_{s,1},x_{s,2},\ldots,x_{s,k_s})^T$
 is the Perron vector of $A_\alpha(G)$,
 where $x_v$ corresponds to $v$, $x_{i,j}$ corresponds to
$u_{ij}$ for $i=1,2,\ldots,s\ \mbox{and}\ j=1,2,\ldots,k_i$, respectively.
Since $A_\alpha(G)X=\lambda_\alpha(G)X$, one can
easily see that
$$\begin{cases}
\ \lambda_\alpha(G)x_{1,i_1}=\alpha x_{1,i_1}+(1-\alpha)x_{1,i_1+1},  & \textrm{$i_1=1,2,\ldots,k_1-1 $,}\\
\ \lambda_\alpha(G)x_{2,i_2}=\alpha x_{2,i_2}+(1-\alpha)x_{2,i_2+1},  & \textrm{$i_2=1,2,\ldots,k_2-1 $,}\\
\ \ \ \ \ \ \ \ \ \ \ \ \  \vdots\\
\ \lambda_\alpha(G)x_{s,i_s}=\alpha x_{s,i_s}+(1-\alpha)x_{s,i_s+1},  & \textrm{$i_s=1,2,\ldots,k_s-1 $,}\\
\ \lambda_\alpha(G)x_v=\alpha s x_v+(1-\alpha)(x_{1,1}+x_{2,1}+\ldots+x_{s,1}),\\
\ \lambda_\alpha(G)x_{j,k_j}=\alpha x_{j,k_j}+(1-\alpha)x_v,  & \textrm{$j=1,2,\ldots,s $}.
\end{cases}$$
Then we have
$$x_{j,k_j}=\left(\frac{\lambda_\alpha(G)-\alpha}{1-\alpha}\right)^{k_j-1}x_{j,1}, \ \ \ \ \ \ \ \ j=1,2,\ldots,s.$$
Furthermore,
$$x_{v}=\left(\frac{\lambda_\alpha(G)-\alpha}{1-\alpha}\right)^{k_j}x_{j,1}, \ \ \ \ \ \ \ \ j=1,2,\ldots,s.$$
Thus, we have
$$\left(\frac{\lambda_\alpha(G)-\alpha s}{1-\alpha}\right)x_v=\left(\frac{\lambda_\alpha(G)-\alpha}{1-\alpha}\right)^{-k_1}x_v+
\left(\frac{\lambda_\alpha(G)-\alpha}{1-\alpha}\right)^{-k_2}x_v+\ldots+\left(\frac{\lambda_\alpha(G)-\alpha}{1-\alpha}\right)^{-k_s}x_v.$$
By the Perron-Frobenius Theorem, we have $x_v>0$, therefore
$$\left(\frac{\lambda_\alpha(G)-s\alpha}{1-\alpha}\right)\left(\frac{\lambda_\alpha(G)-\alpha}{1-\alpha}\right)^{n-1}
=\sum\limits_{i=1}^s\left(\frac{\lambda_\alpha(G)-\alpha}{1-\alpha}\right)^{n-1-k_i}.$$
Let $G'=\widetilde{\theta}(k_1,k_2,\ldots,k_{p-1},k_p-1,k_{p+1},\ldots,k_{q-1},k_q+1,k_{q+1},\ldots,k_s)$. Similarly, we have
\begin{align*}
\left(\frac{\lambda_\alpha(G')-s\alpha}{1-\alpha}\right)\left(\frac{\lambda_\alpha(G')-\alpha}{1-\alpha}\right)^{n-1}&=\sum\limits_{\scriptstyle i=1, \scriptstyle
i\neq p\atop \scriptstyle i\neq q }^s\left(\frac{\lambda_\alpha(G')-\alpha}{1-\alpha}\right)^{n-1-k_i}\\
&+\left(\frac{\lambda_\alpha(G')-\alpha}{1-\alpha}\right)^{n-k_p}+\left(\frac{\lambda_\alpha(G')-\alpha}{1-\alpha}\right)^{n-2-k_q}.
\end{align*}
Let $f(x)=\left(\frac{x-s\alpha}{1-\alpha}\right)\left(\frac{x-\alpha}{1-\alpha}\right)^{n-1}-\sum\limits_{i=1}^s\left(\frac{x-\alpha}{1-\alpha}\right)^{n-1-k_i}$,

$g(x)=\left(\frac{x-s\alpha}{1-\alpha}\right)\left(\frac{x-\alpha}{1-\alpha}\right)^{n-1}-\sum\limits_{\scriptstyle i=1, \scriptstyle
i\neq p\atop \scriptstyle i\neq q }^s\left(\frac{x-\alpha}{1-\alpha}\right)^{n-1-k_i}
-\left(\frac{x-\alpha}{1-\alpha}\right)^{n-k_p}-\left(\frac{x-\alpha}{1-\alpha}\right)^{n-2-k_q}$.

It is easy to see that $\lambda_\alpha(G)$ is the largest
real root of $f(x)=0$. Similarly $\lambda_\alpha(G')$
is the largest real root of $g(x)=0$. Since for all $x>1$
$$
f(x)-g(x)=\left(\left(\frac{x-\alpha}{1-\alpha}\right)-1\right)
\left(\left(\frac{x-\alpha}{1-\alpha}\right)^{n-1-k_p}-\left(\frac{x-\alpha}{1-\alpha}\right)^{n-2-k_q}\right)>0.
$$
Since the
minimum row sum of  $A_\alpha(G')$ is 1, and the row sums of
$A_\alpha(G')$ are not all equal, by Lemma \ref{le:1}, then we
have $\lambda_\alpha(G')>1$. Hence, we get
\begin{align*}
&\lambda_\alpha(\widetilde{\infty}(k_1,k_2,\ldots,k_{p-1},k_p-1,k_{p+1},\ldots,k_{q-1},k_q+1,k_{q+1},\ldots,k_s))\\
&>\lambda_\alpha(\widetilde{\infty}(k_1,k_2,\ldots,k_{p-1},k_p,k_{p+1},\ldots,k_{q-1},k_q,k_{q+1},\ldots,k_s)),
\end{align*}
which prove the result.
\end{proof}
By Lemma \ref{le:14}, we immediately obtain the following theorem.

\noindent\begin{theorem}\label{th:4-2}  Among all
$\widetilde{\infty}$-digraphs on $n$ vertices, the digraph
$\widetilde{\infty}(1,1,1,\ldots,n-s)$ is the unique digraph which attains the maximum $A_\alpha$ spectral radius,
the digraph
$\widetilde{\infty}(a_1,a_2,\ldots,a_s)$ such that
$a_i=\lfloor \frac{n-1}{s} \rfloor$ and $a_j=\lceil \frac{n-1}{s} \rceil$
for any $i \in \{1,2,\ldots,s-(n-1-s\lfloor \frac{n-1}{s} \rfloor)\}$ and
$j \in \{s-(n-1-s\lfloor \frac{n-1}{s} \rfloor)+1,\ldots,s\}$, is the unique digraph which attains the minimum $A_\alpha$ spectral radius.
\end{theorem}

\noindent\begin{lemma}\label{le:12} For any $p, q\in\{1, 2,\ldots,s\}$, if $1\leq k_p\leq k_q$, then we have
\begin{align*}
&\lambda_\alpha(\widetilde{\theta}(k_1,k_2,\ldots,k_{p-1},k_p-1,k_{p+1},\ldots,k_{q-1},k_q+1,k_{q+1},\ldots,k_s,l_1))\\
&>\lambda_\alpha(\widetilde{\theta}(k_1,k_2,\ldots,k_{p-1},k_p,k_{p+1},\ldots,k_{q-1},k_q,k_{q+1},\ldots,k_s,l_1)).
\end{align*}
\end{lemma}
\begin{proof} Let $G=\widetilde{\theta}(k_1,k_2,\ldots,k_{p-1},k_p,k_{p+1},\ldots,k_{q-1},k_q,k_{q+1},\ldots,k_s,l_1)$ be a digraph shown in Figure \ref{fig:d}.
Similar to the proof Lemma \ref{le:14}, we can know that $\lambda_\alpha(G)$ satisfies the follow equation
$$\left(\frac{\lambda_\alpha(G)-s\alpha}{1-\alpha}\right)\left(\frac{\lambda_\alpha(G)-\alpha}{1-\alpha}\right)^{n-1}
=\sum\limits_{i=1}^s\left(\frac{\lambda_\alpha(G)-\alpha}{1-\alpha}\right)^{n-2-l_1-k_i}.$$
Let $G'=\widetilde{\theta}(k_1,k_2,\ldots,k_{p-1},k_p-1,k_{p+1},\ldots,k_{q-1},k_q+1,k_{q+1},\ldots,k_s,l_1)$. Similarly, we have
\begin{align*}
\left(\frac{\lambda_\alpha(G')-s\alpha}{1-\alpha}\right)\left(\frac{\lambda_\alpha(G')-\alpha}{1-\alpha}\right)^{n-1}&=\sum\limits_{\scriptstyle i=1, \scriptstyle
i\neq p\atop \scriptstyle i\neq q }^s\left(\frac{\lambda_\alpha(G')-\alpha}{1-\alpha}\right)^{n-2-l_1-k_i}\\
&+\left(\frac{\lambda_\alpha(G')-\alpha}{1-\alpha}\right)^{n-1-l_1-k_p}+\left(\frac{\lambda_\alpha(G')-\alpha}{1-\alpha}\right)^{n-3-l_1-k_q}.
\end{align*}
Let $f(x)=\left(\frac{x-s\alpha}{1-\alpha}\right)\left(\frac{x-\alpha}{1-\alpha}\right)^{n-1}-\sum\limits_{i=1}^s\left(\frac{x-\alpha}{1-\alpha}\right)^{n-2-l_1-k_i}$,

$g(x)=\left(\frac{x-s\alpha}{1-\alpha}\right)\left(\frac{x-\alpha}{1-\alpha}\right)^{n-1}-\sum\limits_{\scriptstyle i=1, \scriptstyle
i\neq p\atop \scriptstyle i\neq q }^s\left(\frac{x-\alpha}{1-\alpha}\right)^{n-2-l_1-k_i}
-\left(\frac{x-\alpha}{1-\alpha}\right)^{n-1-l_1-k_p}-\left(\frac{x-\alpha}{1-\alpha}\right)^{n-3-l_1-k_q}$.

It is easy to see that $\lambda_\alpha(G)$ is the largest
real root of $f(x)=0$. Similarly $\lambda_\alpha(G')$
is the largest real root of $g(x)=0$. Since for all $x>1$
\begin{align*}
f(x)-g(x)&=\left(\frac{x-\alpha}{1-\alpha}\right)^{n-1-l_1-k_p}-\left(\frac{x-\alpha}{1-\alpha}\right)^{n-2-l_1-k_p}\\
&+\left(\frac{x-\alpha}{1-\alpha}\right)^{n-3-l_1-k_q}-\left(\frac{x-\alpha}{1-\alpha}\right)^{n-2-l_1-k_q}\\
&=\left(\left(\frac{x-\alpha}{1-\alpha}\right)-1\right)\left(\left(\frac{x-\alpha}{1-\alpha}\right)^{n-2-l_1-k_p}-\left(\frac{x-\alpha}{1-\alpha}\right)^{n-3-l_1-k_q}\right)>0.
\end{align*}
Since the
minimum row sum of  $A_\alpha(G')$ is 1, and the row sums of
$A_\alpha(G')$ are not all equal, by Lemma \ref{le:1}, then we
have $\lambda_\alpha(G')>1$. Hence, we get
\begin{align*}
&\lambda_\alpha(\widetilde{\theta}(k_1,k_2,\ldots,k_{p-1},k_p-1,k_{p+1},\ldots,k_{q-1},k_q+1,k_{q+1},\ldots,k_s,l_1))\\
&>\lambda_\alpha(\widetilde{\theta}(k_1,k_2,\ldots,k_{p-1},k_p,k_{p+1},\ldots,k_{q-1},k_q,k_{q+1},\ldots,k_s,l_1)),
\end{align*}
which prove the result.
\end{proof}
Similarly, we have the following lemma.

\noindent\begin{lemma}\label{le:13} If $l_1\geq 1$, then for any $p\in\{1, 2,\ldots,s\}$, then we have
\begin{align*}
&\lambda_\alpha(\widetilde{\theta}(k_1,k_2,\ldots,k_{p-1},k_p+1,k_{p+1},\ldots,k_s,l_1-1))\\
&>\lambda_\alpha(\widetilde{\theta}(k_1,k_2,\ldots,k_{p-1},k_p,k_{p+1},\ldots,k_s,l_1)).
\end{align*}
\end{lemma}
By Lemmas \ref{le:12} and \ref{le:13}, we immediately obtain the following theorem.

\noindent\begin{theorem}\label{th:4-1}  Among all
$\widetilde{\theta}$-digraphs on $n$ vertices, the digraph
$\widetilde{\theta}(0,1,1,\ldots,n-s,0)$ is the unique digraph which attains the maximum $A_\alpha$ spectral radius, the digraph
$\widetilde{\theta}(0,1,1,\ldots,1,n-s-1)$ is the unique digraph which attains the minimum $A_\alpha$ spectral radius.
\end{theorem}

\noindent\begin{lemma}\label{le:4} (\cite{XiWa4}) Let $0\leq \alpha<1$ and $G=(V(G),E(G))$
be a strongly connected digraph on $n$ vertices, $v_p, v_q$ be two distinct vertices of
$V(G)$. Suppose that $v_1,v_2,\ldots,v_t\in N^-_{v_p}\setminus\{N^-_{v_q}\cup\{v_q\}\}$, where $1\leq t\leq d_p^-$,
 and
$X=(x_1,x_2,\ldots,x_n)^T$ be the unique positive unit eigenvector
corresponding to the $A_\alpha$ spectral radius
$\lambda_\alpha(G)$, where $x_i$ corresponds to the vertex $v_i$.
Let $H=G-\{(v_i,v_p): i=1,2\ldots,t\}+\{(v_i,v_q): i=1,2\ldots,t\}$. If $x_q\geq x_p$, then $\lambda_\alpha(H)\geq \lambda_\alpha(G)$.
Furthermore, if $H$ is strongly connected and $x_q>x_p$, then $\lambda_\alpha(H)>
\lambda_\alpha(G)$.
\end{lemma}

\noindent\begin{lemma}\label{le:15} \ For any $\widetilde{\theta}(k_1,k_2,\ldots,k_s,l_1)$-digraph,
there exists $\widetilde{\infty}(k_2,k_3,\ldots,k_s,k_1+l_1+1)$ such that
$$\lambda_\alpha(\widetilde{\theta}(k_1,k_2,\ldots,k_s,l_1))<\lambda_\alpha(\widetilde{\infty}(k_2,k_3,\ldots,k_s,k_1+l_1+1)).$$
\end{lemma}

\begin{proof}
Let $\widetilde{\theta}(k_1,k_2,\ldots,k_s,l_1)$ be a digraph shown in Figure \ref{fig:d} and $X=(x_v,x_u,x_{11},x_{12},$ $\ldots,x_{1k_1},x_{21},x_{22},\ldots,x_{2k_2},\ldots,x_{s1},x_{s2},\ldots,x_{sk_s},y_{1},y_{2},
\ldots,y_{l_1})^T$ be the Perron vector of \\$A_\alpha(\widetilde{\theta}(k_1,k_2,\ldots,k_s,l_1))$, where $x_u$ and $x_v$
correspond to $u$ and $v$, respectively, and $x_{ij}$ correspond to
$w_{ij}$ ($i=1,2,\ldots,s; j=1,2,\ldots,k_i$) and $y_j$ correspond to $u_{1j}$, ($j=1,2,\ldots,l_1$) respectively. It is not difficult to
see that
$\widetilde{\infty}(k_2,k_3,\ldots,k_s,k_1+l_1+1)\cong\widetilde
{\theta}(k_1,k_2,\ldots,k_s,l_1)-\{(w_{2k_2},u),(w_{3k_3},u),\ldots,(w_{sk_{s}},u)\}
+\{(w_{2k_2},v),(w_{3k_3},v),\ldots,(w_{sk_{s}},v)\}$.
Similar to the proof Lemma \ref{le:14}, we have
$$x_v=\left(\frac{\lambda_\alpha(\widetilde{\theta}(k_1,k_2,\ldots,k_s,l_1))-\alpha}{1-\alpha}\right)^{l_1+1}x_u.$$ Since $\lambda_\alpha((\widetilde{\theta}(k_1,k_2,\ldots,k_s,l_1))>1$, we have
$x_v>x_u$. By Lemma \ref{le:4}, we have $\lambda_\alpha(\widetilde{\infty}(k_2,k_3,\ldots,$ $k_{s},k_1+l_1+1))>\lambda_\alpha(\widetilde{\theta}(k_1,k_2,\ldots,k_s,l_1))$. So we complete the proof.
\end{proof}

\noindent\begin{lemma}\label{le:16} \ For any $\widetilde{\infty}(k_1,k_2,\ldots,k_s)$-digraph,
there exists $\widetilde{\theta}(k_1,k_2,\ldots,k_{s-1},k_s-1,0)$ such that
$$\lambda_\alpha(\widetilde{\theta}(k_1,k_2,\ldots,k_{s-1},k_s-1,0))<\lambda_\alpha(\widetilde{\infty}(k_1,k_2,\ldots,k_s)).$$
\end{lemma}
\begin{proof}
 It is not difficult to
see that
$\widetilde{\infty}(k_1,k_2,\ldots,k_s)\cong\widetilde{\theta}(k_1,k_2,\ldots,k_{s-1},k_s-1,0)-\{(w_{1k_1},u),$ $(w_{2k_2},u),\ldots,(w_{s-1k_{s-1}},u)\}
+\{(w_{1k_1},v),(w_{2k_2},v),\ldots,(w_{s-1k_{s-1}},v)\}$.
Similar as the proof Lemma \ref{le:15}, we have $\lambda_\alpha(\widetilde{\infty}(k_1,k_2,\ldots,k_{s-1},k_s))>\lambda_\alpha(\widetilde{\theta}(k_1,k_2,\ldots,k_{s-1},k_s-1,0))$. So we complete the proof.
\end{proof}
By Theorems \ref{th:4-2} and \ref{th:4-1}, Lemmas \ref{le:15} and \ref{le:16}, we immediately obtain the following theorem.

\noindent\begin{theorem}\label{th:4-3}  Among all $\widetilde{\theta}$-digraphs and $\widetilde{\infty}$-digraphs on $n$ vertices, the digraph
$\widetilde{\infty}(1,1,1,$ $\ldots,n-s)$ is the unique digraph which attains the maximum $A_\alpha$ spectral radius, the digraph
$\widetilde{\theta}(0,1,1,\ldots,$ $1,n-s-1)$ is the unique digraph which attains the minimum $A_\alpha$ spectral radius.
\end{theorem}

\noindent\begin{remark}\label{re:4-2} If $s=2$, then the digraph $\widetilde{\infty}(k_1,k_2,\ldots,k_s)$ is $\infty(k_1,k_2)$,
and the digraph $\widetilde{\theta}(k_1,k_2,\ldots,k_s,l_1)$ is $\theta(k_1,k_2,l_1)$. Liu et al. \cite{LWCL} proved that
$\theta(0,1,n-3)$ and $\infty(1,n-2)$ are the digraphs which attain the minimum and maximum $A_\alpha$ spectral radii among
all strongly connected bicyclic digraphs with order $n$, respectively. We generalize their result to $s\geq2$.
\end{remark}

We can know that each strongly connected bicyclic digraph is either a $\theta$-digraph or a $\infty$-digraph. In the following, we
will determine which digraph has the second and the third minimum $A_\alpha$ spectral radius
among all strongly connected bicyclic digraphs, respectively.

\noindent\begin{theorem}\label{th:c-4} \ Among all the strongly
connected bicyclic digraphs with order $n\geq5$, $\theta(1,1,$ $n-4)$ and $\theta(0,2,n-4)$ are the unique digraph which
achieve the second and the third minimum $A_\alpha$ spectral radius, respectively.
\end{theorem}

\begin{proof} Let $G$ be a strongly connected bicyclic digraph with order $n\geq5$ and $G\neq\theta(0,1,n-3)$. Then $G$
is a $\theta$-digraph or a $\infty$-digraph. Suppose that $G$ is a $\theta$-digraph, then $G\neq\theta(0,1,$ $n-3)$, and
by Lemmas \ref{le:12} and \ref{le:13}, we have $\lambda_\alpha(G)\geq\lambda_\alpha(\theta(0,2,n-4))$ with equality only if $G=\theta(0,2,n-4)$ or
$\lambda_\alpha(G)\geq\lambda_\alpha(\theta(1,1,n-4))$ with equality only if $G=\theta(1,1,n-4)$. However, by Lemma \ref{le:12}, we have $\lambda_\alpha(\theta(0,2,n-4))>\lambda_\alpha(\theta(1,1,n-4))$. Thus if $G$ is a $\theta$-digraph and $G\neq\theta(1,1,n-4)$, $\lambda_\alpha(G)\geq\lambda_\alpha(\theta(0,2,n-4))>\lambda_\alpha(\theta(1,1,n-4))$ with equality only if $G=\theta(0,2,n-4)$.
If $G$ is a $\infty$-digraph, then by Lemma \ref{le:14},
$\lambda_\alpha(G)\geq \lambda_\alpha(\infty(\lfloor{\frac{n-1}{2}}\rfloor,\lceil{\frac{n-1}{2}}\rceil))$. If $n$ is odd, $\frac{n-1}{2}\geq2$,
then by Lemmas \ref{le:16}, \ref{le:12} and \ref{le:13}, we have
$\lambda_\alpha(\infty(\lfloor{\frac{n-1}{2}}\rfloor,\lceil{\frac{n-1}{2}}\rceil))=
\lambda_\alpha(\infty(\frac{n-1}{2},\frac{n-1}{2}))>\lambda_\alpha(\theta(\frac{n-3}{2},\frac{n-1}{2},0))>\lambda_\alpha(\theta(0,2,n-4))$.
If $n$ is even, $\frac{n-2}{2}\geq2$, then by Lemmas \ref{le:16}, \ref{le:12} and \ref{le:13}, we have
$\lambda_\alpha(\infty(\lfloor{\frac{n-1}{2}}\rfloor,\lceil{\frac{n-1}{2}}\rceil))=
\lambda_\alpha(\infty(\frac{n-2}{2},\frac{n}{2}))>\lambda_\alpha(\theta(\frac{n-2}{2},\frac{n-2}{2},0))>\lambda_\alpha(\theta(0,2,n-4))$.
Hence, if $G$ is a $\infty$-digraph, then we have $$\lambda_\alpha(G)\geq \lambda_\alpha\left(\infty\left(\left\lfloor{\frac{n-1}{2}}\right\rfloor,\left\lceil{\frac{n-1}{2}}\right\rceil\right)\right)>\lambda_\alpha(\theta(0,2,n-4)).$$
Therefore, by the second part of Theorem \ref{th:4-3}, we get the result.
\end{proof}

\section{The second, the third and the forth minimum $A_\alpha$ spectral radius of strongly connected digraphs}

In the followig, we determine the digraphs which achieve the second, the third and the
forth minimum $A_\alpha$ spectral radius of strongly connected digraphs on $n$ vertices.

Recall that the spectral radius of a nonnegative irreducible matrix $B$ is larger than
that of a principal submatrix of $B$ and it increases when an entry of $B$ increases \cite{BP}. Thus we have the following well known lemma.

\noindent\begin{lemma}\label{le:9} Let $G$ be a strongly connected digraph and $H$ a  proper
subdigraph of $G$. Then $\lambda_\alpha(G)>\lambda_\alpha(H)$.
\end{lemma}

\noindent\begin{corollary}\label{co:1} Let $G$ be a strongly connected digraph. Then $1\leq\lambda_\alpha(G)\leq n-1$, $\lambda_\alpha(G)=1$ if and
only if $G\cong C_{n}$, and $\lambda_\alpha(G)=n-1$ if and
only if $G\cong\overset{\longleftrightarrow}{K_{n}}$.
\end{corollary}

%\noindent\begin{lemma}\label{le:10}
%$\lambda_\alpha(\theta(0,2,n-4)$ is strictly decreasing for $n\geq4$.
%\end{lemma}

%\begin{proof} Suppose that $n_2>n_1\geq4$. By Lemma \ref{le:2}, we have $\lambda_\alpha(\theta(0,2,n_1-4)$ is the largest root of $f(x)=(x-2\alpha)(\frac{x-\alpha}{1-\alpha})^{n_1-1}-(1-\alpha)(\frac{x-\alpha}{1-\alpha})^2-(1-\alpha)=0$.
%Similarly we have $\lambda_\alpha(\theta(0,2,n_2-4)$ is the largest root of $g(x)=(x-2\alpha)(\frac{x-\alpha}{1-\alpha})^{n_2-1}-(1-\alpha)(\frac{x-\alpha}{1-\alpha})^2-(1-\alpha)=0$.
%Then $g(x)-f(x)=(x-2\alpha)(\frac{x-\alpha}{1-\alpha})^{n_2-1}-(x-2\alpha)(\frac{x-\alpha}{1-\alpha})^{n_1-1}
%=(x-2\alpha)(\frac{x-\alpha}{1-\alpha})^{n_2-1}-\frac{x-\alpha}{1-\alpha})^{n_1-1})$. Since, by Lemma \ref{le:1}, we
%have $\lambda_\alpha(\theta(a,b,c))>1$, and by the proof of Lemma \ref{le:2}, we have $\lambda_\alpha(\theta(a,b,c))>2\alpha$. Hence,
%$g(x)-f(x)>0$ for all $x>\lambda_\alpha(\theta(0,2,n_1-4))$. Thus $\lambda_\alpha(\theta(0,2,n_1-4))>\lambda_\alpha(\theta(0,2,n_2-4))$
%\end{proof}

\noindent\begin{lemma}\label{le:11} (\cite{XiWa4}) Let $0\leq \alpha<1$ and $G$ ($\neq C_{n}$) be a strongly connected digraph with $V(G)=\{v_1,v_2,\cdots,v_n\}$, $(v_i,v_j)\in E(G)$ and $w\notin V(G)$, $G^w=(V(G^w),E(G^w))$ with $V(G^w)=V(G)\cup\{w\}$,
$E(G^w)=E(D)-\{(v_i,v_j)\}+\{(v_i,w),(w,v_j)\}$. Then
$\lambda_\alpha(G)\geq\lambda_\alpha(G^w)$.
\end{lemma}

We follow the techniques in \cite{LZ} to prove the following result.

\noindent\begin{theorem}\label{th:c-5} \ Let $0\leq \alpha\leq\frac{1}{2}$ and $G$ be a strongly
connected digraph of order $n\geq5$ that is neither a bicyclic digraph nor $C_{n}$, Then
$\lambda_\alpha(G)>\lambda_\alpha(\theta(0,2,n-4))$.
\end{theorem}

\begin{proof} Let $C$ be a shortest directed cycle in $G$. Obviously, $V(C)\neq V(G)$. There is a vertex $u\in V(G)\setminus V(C)$ such that there is
 a arc from $u$ to some vertex, say $v$ on $C$. Also, there is a directed path from some vertex on $C$ to $u$. Let $w$ be a vertex on $C$
 such that the distance from $w$ to $u$ is as small as possible. Let $P$ be such a directed path. Then $P$ and $C$ have exactly one common vertex $w$.
 If $w=v$, then $G$ has a proper $\infty$-subdigraph.  If $w\neq v$, then $G$ has a proper $\theta$-subdigraph.

{\bf Case 1.}  If $G$ has a proper $\infty$-subdigraph, say $\infty(k_1,l_1)$ with $k_1+l_1=n_1-1$ and $n_1\leq n$, then by Lemma \ref{le:9}, the second part
of Theorem \ref{th:4-2}, and  Theorem \ref{th:c-4}, Lemma \ref{le:11}, we have
\begin{align*}
&\lambda_\alpha(G)>\lambda_\alpha(\infty(k_1,l_1))\\
&\geq\lambda_\alpha\left(\infty\left(\left\lfloor{\frac{n_1+1}{2}}\right\rfloor,\left\lceil{\frac{n_1+1}{2}}\right\rceil\right)\right)\\
&>\lambda_\alpha(\theta(0,2,n_1-4))\\
&\geq\lambda_\alpha(\theta(0,2,n-4)).
\end{align*}

{\bf Case 2.}  If $G$ has a proper $\theta$-subdigraph, say $\theta(a_1,b_1,c_1)$ with $a_1+b_1+c_1=n_2-2$ and $n_2\leq n$.

{\bf Subase 2.1.} $n_2\leq n-1$. By Lemma \ref{le:9}, the second part of Theorem \ref{th:4-1} and Lemma \ref{le:11}, we get
$$\lambda_\alpha(G)>\lambda_\alpha(\theta(a_1,b_1,c_1))\geq\lambda_\alpha(\theta(0,1,n_2-3))\geq\lambda_\alpha(\theta(0,1,n-4))\geq\lambda_\alpha(\theta(0,2,n-4)).$$

{\bf Subase 2.2.} $n_2=n$ and $\theta(a_1,b_1,c_1)\neq\theta(0,1,n-3)$ and $\theta(a_1,b_1,c_1)\neq\theta(1,1,n-4)$.
By Lemmas \ref{le:9}, the second part of Theorem \ref{th:4-1}, and Theorem \ref{th:c-4}, we get
$$\lambda_\alpha(G)>\lambda_\alpha(\theta(a_1,b_1,c_1))\geq\lambda_\alpha(\theta(0,2,n-4)).$$

{\bf Subase 2.3.} $n_2=n$ and the $\theta$-subdigraph of $G$ can only be $\theta(0,1,n-3)$ or $\theta(1,1,n-4)$.

{\bf Subase 2.3.1.} If $G$ has a
$\theta$-subdigraph $\theta(0,1,n-3)$. Let $wv$, $wuv$ and $vu_1u_2\ldots u_{n-3}w$ be the basic directed paths of the $\theta$-subdigraph $\theta(0,1,n-3)$.
We consider the possible arc(s) in $G$ except the arcs in $\theta(0,1,n-3)$ as follows.

(1) $(v,w)\notin E(G)$, otherwise, $G$ has a
$\theta$-subdigraph $\theta(0,n-3,0)$, a contradiction.

(2) $(v,u)\notin E(G)$ and $(u,w)\notin E(G)$, otherwise, $G$ has a
$\theta$-subdigraph $\theta(0,n-2,0)$, a contradiction.

(3) $(u,u_k)\notin E(G)$ and $(u_{n-k-2},u)\notin E(G)$ for $2\leq k\leq n-3$, otherwise, $G$ has a
$\theta$-subdigraph $\theta(0,k,n-k-2)$, a contradiction.

(4) $(w,u_k)\notin E(G)$ and $(u_{n-k-2},v)\notin E(G)$ for $1\leq k\leq n-3$, otherwise, $G$ has a
$\theta$-subdigraph $\theta(0,k+1,n-k-3)$, a contradiction.

(5) $(u_k,w)\notin E(G)$ and $(v,u_{n-k-2})\notin E(G)$ for $1\leq k\leq n-4$, otherwise, $G$ has a
$\theta$-subdigraph $\theta(0,1,k)$, a contradiction.

(6) $(u_l,u_k)\notin E(G)$ for $1\leq k<l\leq n-3$, otherwise, $G$ has a
$\theta$-subdigraph $\theta(0,n-l+k-1,l-k-1)$, a contradiction.

(7) $(u_k,u_l)\notin E(G)$ for $1\leq k<l-1\leq n-4$, otherwise, $G$ has a
$\theta$-subdigraph $\theta(0,1,n-2+k-l)$, a contradiction.

(8) $\{(u,u_1),(u_{n-3},u)\}\notin E(G)$, otherwise, $G$ has a
$\theta$-subdigraph $\theta(0,1,n-4)$, a contradiction.

From (1)-(8), we find that besides these arcs in $\theta(0,1,n-3)$, $G$ contains one additional arc
$(u,u_1)$ or $(u_{n-3},u)$. Thus $G$ is isomorphic to the digraph $G'$ obtained from $\theta(0,1,n-3)$ by adding the arcs $(u,u_1)$, as shown
in the Figure \ref{Figure 3.}.
\begin{figure}[htbp]
\begin{centering}
%\subfigure[$G_1$]
%{\label{G_1}
{\includegraphics[scale=0.6]{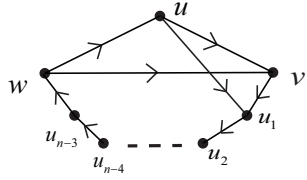}}
\caption{The digraph $G'$.}\label{Figure 3.}
\end{centering}
\end{figure}

Similar to the proofs of Lemmas \ref{le:14} and \ref{le:12}, we have $\lambda_\alpha(G')$ is the largest
real root of $p(x)=(\frac{x-2\alpha}{1-\alpha})^2(\frac{x-\alpha}{1-\alpha})^{n-2}-\frac{2x-3\alpha}{1-\alpha}-1=0$. From the proof of
Lemma \ref{le:12}, we know that $\lambda_\alpha(\theta(0,2,n-4))$ is the largest
real root of $q(x)=\frac{x-2\alpha}{1-\alpha}(\frac{x-\alpha}{1-\alpha})^{n-1}-(\frac{x-\alpha}{1-\alpha})^{2}-1=0$.
Note that $$q(x)-p(x)=\frac{x-2\alpha}{1-\alpha}\left(\frac{x-\alpha}{1-\alpha}\right)^{n-2}\frac{\alpha}{1-\alpha}-\left(\frac{x-\alpha}{1-\alpha}\right)^{2}+\frac{2x-3\alpha}{1-\alpha}.$$
For $0\leq \alpha\leq\frac{1}{2}$,
\begin{align*}
q(x)-p(x)&>\frac{x-2\alpha}{1-\alpha}\left(\frac{x-\alpha}{1-\alpha}\right)\frac{\alpha}{1-\alpha}-\frac{x^2-2x+3\alpha-2\alpha^2}{(1-\alpha)^{2}}, \\
&=\frac{\alpha x^2-3\alpha^2x+2\alpha^3}{(1-\alpha)^{3}}-\frac{(x^2-2x+3\alpha-2\alpha^2)(1-\alpha)}{(1-\alpha)^{3}}\\
&=\frac{(2\alpha-1)x^2+(2-2\alpha-3\alpha^2)x-3\alpha+5\alpha^2}{(1-\alpha)^{3}}.
\end{align*}
Taking $g(x)=(2\alpha-1)x^2+(2-2\alpha-3\alpha^2)x-3\alpha+5\alpha^2$. If $\alpha=\frac{1}{2}$, then $g(x)=\frac{1}{4}(x-1)$. Thus
$g(x)>0$ for all $x>1$. Then $q(x)-p(x)>0$ for all $x>1$. However, by Lemma \ref{le:1}, we
have $\lambda_\alpha(G')>1$, Then, we get $\lambda_\alpha(G)=\lambda_\alpha(G')>\lambda_\alpha(\theta(0,2,n-4))$. If $0\leq \alpha<\frac{1}{2}$,
then $2\alpha-1<0$, and $g(x)''<0$  for $1<x<2$. Hence
$g(x)>\min\{g(1),g(2)\}=\min\{1-3\alpha+2\alpha^2,\alpha-\alpha^2\}\geq0$ for $0\leq \alpha<\frac{1}{2}$. Hence $q(x)-p(x)>0$ for all $1<x<2$.
However, by Lemma \ref{le:1}, we have $1<\lambda_\alpha(G')<2$. Then, we have $\lambda_\alpha(G)=\lambda_\alpha(G')>\lambda_\alpha(\theta(0,2,n-4))$.

{\bf Subase 2.3.2.} If $G$ has a
$\theta$-subdigraph $\theta(1,1,n-4)$. Let $uwv$, $uw_1v$ and $vw_1'w_2'\ldots u_{n-4}'u$ be the basic directed paths of the $\theta$-subdigraph $\theta(1,1,n-4)$.
We consider the possible arc(s) in $G$ except the arcs in $\theta(1,1,n-4)$ as follows.

(1) $(w,u)\notin E(G)$ and $(v,w)\notin E(G)$, otherwise, $G$ has a
$\theta$-subdigraph $\theta(0,n-3,0)$, a contradiction.

(2) $(w_1,u)\notin E(G)$ and $(v,w_1)\notin E(G)$, otherwise, $G$ has a
$\theta$-subdigraph $\theta(0,n-3,0)$, a contradiction.

(3) $(v,u)\notin E(G)$, otherwise, $G$ has a
$\theta$-subdigraph $\theta(0,n-4,1)$, a contradiction.

(4) $(u,v)\notin E(G)$, otherwise, $G$ has a
$\theta$-subdigraph $\theta(0,1,n-4)$, a contradiction.

(5) $(w,w_k')\notin E(G)$ and $(w_{n-k-3}',w)\notin E(G)$ for $1\leq k\leq n-4$, otherwise, $G$ has a
$\theta$-subdigraph $\theta(0,k,n-k-3)$, a contradiction.

(6) $(w_1,w_k')\notin E(G)$ and $(w_{n-k-3}',w_1)\notin E(G)$ for $1\leq k\leq n-4$, otherwise, $G$ has a
$\theta$-subdigraph $\theta(0,k,n-k-3)$, a contradiction.

(7) $(v,w_k')\notin E(G)$ for $2\leq k\leq n-4$, otherwise, $G$ has a
$\theta$-subdigraph $\theta(0,k-1,n-k-2)$, a contradiction.

(8) $(w_k',v)\notin E(G)$ for $1\leq k\leq n-4$, otherwise, $G$ has a
$\theta$-subdigraph $\theta(0,n-k-2,k-1)$, a contradiction.

(9) $(u,w_k')\notin E(G)$ for $1\leq k\leq n-4$, otherwise, $G$ has a
$\theta$-subdigraph $\theta(0,k+1,n-k-4)$, a contradiction.

(10) $(w_k',u)\notin E(G)$ for $1\leq k\leq n-5$, otherwise, $G$ has a
$\theta$-subdigraph $\theta(0,n-k-4, k+1)$, a contradiction.

(11) $(w_l',w_k')\notin E(G)$ for $1\leq k<l\leq n-4$, otherwise, $G$ has a
$\theta$-subdigraph $\theta(0,n-l+k-2,l-k-1)$, a contradiction.

(12) $(w_k',w_l')\notin E(G)$ for $1\leq k<l-1\leq n-5$, otherwise, $G$ has a
$\theta$-subdigraph $\theta(1,1,n-3+k-l)$, a contradiction.

(13) $\{(w,w_1),(w_1,w)\}\notin E(G)$, otherwise, $G$ has a
$\theta$-subdigraph $\theta(0,n-2,0)$, a contradiction.

From (1)-(13), we find that besides these arcs in $\theta(1,1,n-4)$, $G$ only contains one additional arc
$(w,w_1)$ or $(w_1,w)$. Thus $G$ is isomorphic to the digraph $G_1$ or $G_2$, where $G_1$ and $G_2$
as shown in the Figure \ref{Figure 4.}.
\begin{figure}[htbp]
\begin{centering}
%\subfigure[$G_1$]
%{\label{G_1}
{\includegraphics[scale=0.6]{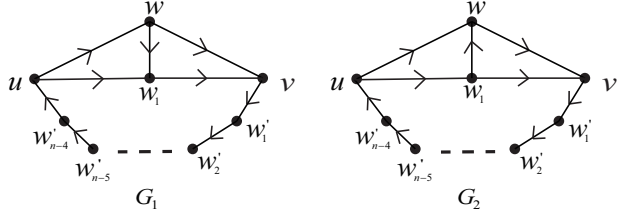}}
\caption{The digraphs $G_1$ and $G_2$.}\label{Figure 4.}
\end{centering}
\end{figure}
If $G$ is isomorphic to the digraph $G_1$, one can easily get that $\lambda_\alpha(G_1)$ is the largest
real root of the equation $(\frac{x-2\alpha}{1-\alpha})^2(\frac{x-\alpha}{1-\alpha})^{n-2}-\frac{2x-3\alpha}{1-\alpha}-1=0$.
From the proof of subcase 2.3.1, we have $\lambda_\alpha(G_1)=\lambda_\alpha(G')>\lambda_\alpha(\theta(0,2,n-4))$. Thus
we have $\lambda_\alpha(G)=\lambda_\alpha(G_1)=\lambda_\alpha(G')>\lambda_\alpha(\theta(0,2,n-4))$.

If $G$ is isomorphic to the digraph $G_2$, note that $G_2$ isomorphic to the digraph $G'$ as shown in subcase 2.3.1. Thus
we have $\lambda_\alpha(G)=\lambda_\alpha(G_2)=\lambda_\alpha(G')>\lambda_\alpha(\theta(0,2,n-4))$.

Combining the above two cases, we have $\lambda_\alpha(G)>\lambda_\alpha(\theta(0,2,n-4))$, if $G$ is a strongly
connected digraph of order $n\geq5$ that is neither a bicyclic digraph nor $C_{n}$
\end{proof}

By Corollary \ref{co:1}, we know that $C_{n}$ is the unique
strongly connected digraph with the minimum $A_\alpha$ spectral radius among all the strongly
connected digraphs of order $n$. Therefore, from Theorems \ref{th:4-3}, \ref{th:c-4} and \ref{th:c-5}, we have the following theorem.

\noindent\begin{theorem}\label{th:c-6} \ Among all the strongly
connected digraphs with order $n\geq5$ and $0\leq \alpha\leq\frac{1}{2}$, $\theta(0,1,n-3)$, $\theta(1,1,n-4)$ and $\theta(0,2,n-4)$ are the digraphs which
achieve the second, the third and the fourth minimum $A_\alpha$ spectral radius, respectively.
\end{theorem}

\noindent\begin{remark}\label{re:2} If $\alpha=0$, Li and Zhou \cite{LZ} proved that $\theta(0,1,n-3)$, $\theta(1,1,n-4)$ and $\theta(0,2,n-4)$ are the unique digraphs which
achieve the second, the third and the fourth minimum $A_0$ spectral radius among all strongly connected digraphs, respectively.
If $\alpha=\frac{1}{2}$, Hong and You \cite{HoYo} determined that $\theta(0,1,n-3)$, $\theta(1,1,n-4)$ and $\theta(0,2,n-4)$ also attain the second, the third and the fourth minimum $A_{\frac{1}{2}}$ spectral radii among all strongly connected digraphs, respectively.
\end{remark}

For general $0\leq \alpha<1$, we propose the following conjecture based on numerical examples.

\noindent\begin{conjecture} \ Among all the strongly
connected digraphs with order $n\geq5$ and $0\leq \alpha<1$, $\theta(0,1,n-3)$, $\theta(1,1,n-4)$ and $\theta(0,2,n-4)$ are the digraphs which
achieve the second, the third and the fourth minimum $A_\alpha$ spectral radius, respectively.
\end{conjecture}
%\noindent\begin{lemma}\label{le:3} (\cite{HoYo})
%The digraph which achieved the maximum signless Laplacian spectral radius
% all strongly connected digraphs with vertex connectivity $1\leq\eta(G)=k\leq n-2$
% must be some digraph in $K(n,k)$.
%\end{lemma}
\section{ The $A_\alpha$ spectral radius of strongly connected bipartite digraphs which contain a complete bipartite subdigraph}

Let $\overleftrightarrow{K_{p,q}}$ be a complete bipartite digraph with
$V(\overleftrightarrow{K_{p,q}})=V_p \cup \ V_q$ and $|V_p|=p$, $|V_q|=q$.
Let $\mathcal{G}_{n,p,q}$ denote the set of strongly connected
bipartite digraphs on $n$ vertices which contain a complete bipartite subdigraph $\overleftrightarrow{K_{p,q}}$.
As we all know, if $p+q=n$, then $\mathcal{G}_{n,p,q}=\{\overleftrightarrow{K_{p,q}}\}$. It is easy to know that
$\lambda_\alpha(\overleftrightarrow{K_{p,q}})=\frac{\alpha(p+q)+\sqrt{(\alpha(p+q))^2-8\alpha pq+4pq}}{2}$.
Thus we only consider the cases when $p+q\leq n-1$ and $p\geq q\geq2$. In the rest of this section, we just discuss under this assumption.

Chen et al. \cite{CCL} proved that if $n\equiv p+q \ (mod \ 2)$ then $B^5_{n,p,q}$ or $B^6_{n,p,q}$
 is the unique bipartite digraph with the minimum $A_0$ spectral radius among all digraphs in $\mathcal{G}_{n,p,q}$,
otherwise, if $n\not\equiv p+q \ (mod \ 2)$ then $B^1_{n,p,q}$ is the unique bipartite digraph with
the minimum $A_0$ spectral radius among all digraphs in $\mathcal{G}_{n,p,q}$. We generalize their results to $0\leq\alpha<1$.

Let $B^1_{n,p,q}$ be a digraph obtained by adding a directed path
$P_{n-p-q+2}=v_1v_{p+q+1}v_{p+q+2}\ldots$ $v_{n}v_{p}$ to a complete bipartite digraph
$\overleftrightarrow{K_{p,q}}$ such that $V(\overleftrightarrow{K_{p,q}}) \cap V(P_{n-p-q+2})=\{v_1,v_p\}$
as shown in Figure \ref{Fig.1}(a), where $V(B^1_{n,p,q})=\{v_1,v_2,\ldots,v_n\}$. Clearly, if $n-p-q$ is odd,
then $B^1_{n,p,q}\in \mathcal{G}_{n,p,q}$.

Let $B^2_{n,p,q}$ be a digraph obtained by adding a directed path
$P_{n-p-q+2}=v_{p+1}v_{p+q+1}v_{p+q+2}$ $\ldots v_{n}v_{p+q}$ to a complete bipartite digraph
$\overleftrightarrow{K_{p,q}}$ such that $V(\overleftrightarrow{K_{p,q}}) \cap V(P_{n-p-q+2})=\{v_{p+1},$ $v_{p+q}\}$
as shown in Figure \ref{Fig.1}(b), where $V(B^2_{n,p,q})=\{v_1,v_2,$ $\ldots,v_n\}$. Clearly, if $n-p-q$ is odd,
then $B^2_{n,p,q}\in \mathcal{G}_{n,p,q}$.

\begin{figure}[htbp]
\centering
\includegraphics[scale=0.5]{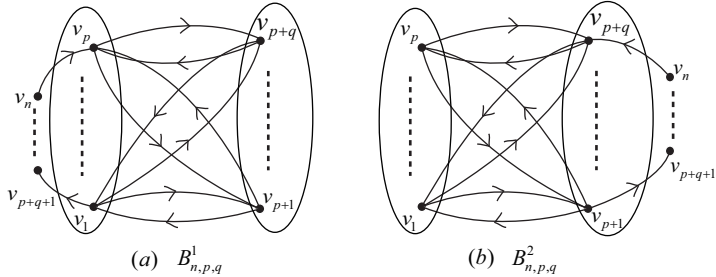}
\caption{$B^1_{n,p,q}$ and $B^2_{n,p,q}$}\label{Fig.1}
\end{figure}

Let $B^5_{n,p,q}$ be a digraph obtained by adding a directed path
$P_{n-p-q+2}=v_1v_{p+q+1}v_{p+q+2}$ $\ldots v_{n}v_{p+1}$ to a complete bipartite digraph
$\overleftrightarrow{K_{p,q}}$ such that $V(\overleftrightarrow{K_{p,q}}) \cap V(P_{n-p-q+2})=\{v_1,v_{p+1}\}$
as shown in Figure \ref{Fig.2}(a), where $V(B^5_{n,p,q})=\{v_1,v_2,\ldots,$ $v_n\}$. Clearly, if $n-p-q$ is even, then $B^5_{n,p,q}\in \mathcal{G}_{n,p,q}$.

Let $B^6_{n,p,q}$ be a digraph obtained by adding a directed path
$P_{n-p-q+2}=v_{p+1}v_{p+q+1}v_{p+q+2}$ $\ldots v_{n}v_1$ to a complete bipartite digraph
$\overleftrightarrow{K_{p,q}}$ such that $V(\overleftrightarrow{K_{p,q}}) \cap V(P_{n-p-q+2})=\{v_1,v_{p+1}\}$
as shown in Figure \ref{Fig.2}(b), where $V(B^6_{n,p,q})=\{v_1,v_2,\ldots,$ $v_n\}$. Clearly, if $n-p-q$ is even, then $B^6_{n,p,q}\in \mathcal{G}_{n,p,q}$.

\begin{figure}[htbp]
\centering
\includegraphics[scale=0.5]{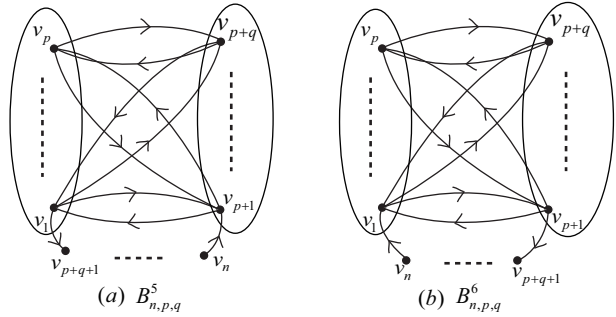}
\caption{$B^5_{n,p,q}$ and $B^6_{n,p,q}$}\label{Fig.2}
\end{figure}

\noindent\begin{lemma}\label{le:3} (\cite{LWCL}) Let $0\leq \alpha<1$ and $G$ be a strongly connected digraph. Then
$\lambda_\alpha(G)> \alpha\Delta^+$, where $\Delta^+$ denotes the maximum outdegree of $G$.
\end{lemma}

\noindent\begin{lemma}\label{le:a}  Let $G$ be a strongly connected digraph containing two vertices $v_i, v_j$ such that $(v_i,v_j)\notin E(G)$
and $(v_j,v_i)\notin E(G)$, and let $X$ be the Perron vector of $A_\alpha(G)$, where $x_i$ corresponds to the vertex $v_i$.
If $N^+_i\subseteq N^+_j$, then $x_j\geq x_i$.
Moreover, if $N^+_i\subsetneq N^+_j$, then $x_j>x_i$, if $N^+_i=N^+_j$, then $x_j=x_i$.
\end{lemma}

\begin{proof} From $A_\alpha(G)X=\lambda_\alpha(G)X$, we have

$$\lambda_\alpha(G) x_i=\alpha d^+_ix_i+(1-\alpha)\sum\limits_{v_k\in N^+_i}x_k,$$
$$\lambda_\alpha(G) x_j=\alpha d^+_jx_j+(1-\alpha)\sum\limits_{v_k\in N^+_j}x_k,$$
Since $(v_i,v_j)\notin E(G)$ and $(v_j,v_i)\notin E(G)$, and $N^+_i\subseteq N^+_j$, we have $d^+_i\leq d^+_j$.
Furthermore, we get $(\lambda_\alpha(G)-\alpha d^+_j)x_j\geq(\lambda_\alpha(G)-\alpha d^+_i)x_i$. By Lemma \ref{le:3},
$\lambda_\alpha(G)> \alpha\Delta^+$. So $x_j\geq x_i$.

Since $v_j\notin N^+_i$, $v_i\notin N^+_j$, if $N^+_i\subsetneq N^+_j$, then $d^+_i< d^+_j$ and
$(\lambda_\alpha(G)-\alpha d^+_j)x_j>(\lambda_\alpha(G)-\alpha d^+_i)x_i$, which implies $x_j> x_i$, and
if $N^+_i= N^+_j$, then $d^+_i= d^+_j$ and $(\lambda_\alpha(G)-\alpha d^+_j)x_j= (\lambda_\alpha(G)-\alpha d^+_i)x_i$, which implies $x_i=x_j$.
\end{proof}

\noindent\begin{theorem}\label{th:ch-1} For digraphs $B^1_{n,p,q}$ and $B^2_{n,p,q}$, as shown in Figure \ref{Fig.1},
$$\lambda_\alpha(B^2_{n,p,q})\geq\lambda_\alpha(B^1_{n,p,q}),$$ with equality if and only if $p=q$.
\end{theorem}

\begin{proof} If $p=q$, then $B^2_{n,p,q}\cong B^1_{n,p,q}$. Hence $\lambda_\alpha(B^1_{n,p,q})=\lambda_\alpha(B^2_{n,p,q})$. Otherwise $p>q$, let $G=B^1_{n,p,q}$ and
$X=(x_1,x_2,\ldots,x_n)^T$ be the
Perron vector corresponding to $\lambda_\alpha(G)$ where $x_i$ corresponds to the vertex $v_i$.
By Lemma \ref{le:a}, $x_2=x_3=\ldots=x_p\triangleq x_p$ and $x_{p+1}=x_{p+2}=\ldots=x_{p+q}\triangleq x_{p+1}$.
From $A_\alpha(G)X=\lambda_\alpha(G)X$, we have
$$\begin{cases}
\ \lambda_\alpha(G)x_1=\alpha(q+1)x_1+(1-\alpha)x_{p+q+1}+(1-\alpha)qx_{p+1},\\
\ \lambda_\alpha(G)x_p=\alpha qx_p+(1-\alpha)qx_{p+1},\\
\ \lambda_\alpha(G)x_{p+1}=\alpha px_{p+1}+(1-\alpha)x_1+(1-\alpha)(p-1)x_p,\\
\ \lambda_\alpha(G)x_i=\alpha x_i+(1-\alpha)x_{i+1},&\textrm{$i=p+q+1,\ldots,n-1$,}\\
\ \lambda_\alpha(G)x_n=\alpha x_n+(1-\alpha)x_p.
\end{cases}$$
Then
$$x_n=\left(\frac{\lambda_\alpha(G)-\alpha}{1-\alpha}\right)^{n-p-q-1}x_{p+q+1},$$
$$x_p=\left(\frac{\lambda_\alpha(G)-\alpha}{1-\alpha}\right)x_n=\left(\frac{\lambda_\alpha(G)-\alpha}{1-\alpha}\right)^{n-p-q}x_{p+q+1}.$$
Therefore,
\begin{align*}
(\lambda_\alpha(G)-\alpha q)(\lambda_\alpha(G)-\alpha p)x_p&=(1-\alpha)\cdot q\cdot(\lambda_\alpha(G)-\alpha p)x_{p+1}\\
&=(1-\alpha)^2\cdot qx_1+(1-\alpha)^2\cdot q\cdot (p-1)x_p.
\end{align*}
Furthermore
\begin{align*}
&(\lambda_\alpha(G)-\alpha q)\cdot(\lambda_\alpha(G)-\alpha p)\cdot(\lambda_\alpha(G)-\alpha (q+1))x_p\\
&=(1-\alpha)^2\cdot q\cdot (\lambda_\alpha(G)-\alpha (q+1))x_1+(1-\alpha)^2\cdot q\cdot(p-1)(\lambda_\alpha(G)-\alpha (q+1))x_p\\
&=(1-\alpha)^2\cdot q\cdot((1-\alpha)x_{p+q+1}+(1-\alpha)\cdot qx_{p+1})+(1-\alpha)^2\cdot q\cdot(p-1)(\lambda_\alpha(G)-\alpha (q+1))x_p\\
&=(1-\alpha)^3\cdot q\cdot\frac{1}{\left(\frac{\lambda_\alpha(G)-\alpha}{1-\alpha}\right)^{n-p-q}}x_p+(1-\alpha)^2\cdot q\cdot (\lambda_\alpha(G)-\alpha q)x_p\\
&\ \ \ \ \ \ +(1-\alpha)^2\cdot q\cdot(p-1)\cdot(\lambda_\alpha(G)-\alpha (q+1))x_p.
\end{align*}
Note that $x_p>0$. Hence
\begin{align*}
&\left(\frac{\lambda_\alpha(G)-\alpha}{1-\alpha}\right)^{n-p-q}[(\lambda_\alpha(G)-\alpha q)\cdot(\lambda_\alpha(G)-\alpha p)\cdot(\lambda_\alpha(G)-\alpha (q+1))\\
&\ \ \ \ -(1-\alpha)^2\cdot q\cdot (\lambda_\alpha(G)-\alpha q)-(1-\alpha)^2\cdot q\cdot(p-1)(\lambda_\alpha(G)-\alpha (q+1))]-(1-\alpha)^3 q=0.
\end{align*}

Let $f(x)=\left(\frac{x-\alpha}{1-\alpha}\right)^{n-p-q}[x^3-(\alpha p+2\alpha q+\alpha)x^2+(\alpha^2 q^2+\alpha^2pq+2\alpha pq+\alpha^2q+\alpha^2p-pq)x
-2\alpha^2q^2p-2\alpha^2pq+\alpha q^2p+\alpha pq+2\alpha^2q-\alpha^3q-\alpha q]-(1-\alpha)^3q$.
It is not difficult to see that $\lambda_\alpha(B^1_{n,p,q})$ is the largest
real root of $f(x)=0$. Similarly,
let $g(x)=\left(\frac{x-\alpha}{1-\alpha}\right)^{n-p-q}[x^3-(\alpha q+2\alpha p+\alpha)x^2+(\alpha^2 p^2+\alpha^2pq+2\alpha pq+\alpha^2p+\alpha^2q-pq)x
-2\alpha^2p^2q-2\alpha^2pq+\alpha p^2q+\alpha pq+2\alpha^2p-\alpha^3p-\alpha p]-(1-\alpha)^3p$, then
$\lambda_\alpha(B^2_{n,p,q})$ is the largest real root of $g(x)=0$. Thus
\begin{align*}
f(x)-g(x)&=\left(\frac{x-\alpha}{1-\alpha}\right)^{n-p-q}[\alpha(p-q)x^2-\alpha^2(p-q)(p+q)x+\alpha(p-q)\\
&\ \ \ \ -\alpha pq(p-q)+2\alpha^2pq(p-q)-2\alpha^2(p-q)+\alpha^3(p-q)]+(1-\alpha)^3(p-q)\\
&=\left(\frac{x-\alpha}{1-\alpha}\right)^{n-p-q}\cdot\alpha\cdot(p-q)\cdot[x^2-\alpha(p+q)x+1-pq+2\alpha pq-2\alpha+\alpha^2]\\
&\ \ \ \  +(1-\alpha)^3(p-q).
\end{align*}
Since $\lambda_\alpha(\overleftrightarrow{K_{p,q}})$ is the the largest real root of the equation $x^2-\alpha(p+q)x-pq+2\alpha pq=0$,
$x^2-\alpha(p+q)x-pq+2\alpha pq>0$ for all $x>\lambda_\alpha(\overleftrightarrow{K_{p,q}})$.
Thus $x^2-\alpha(p+q)x+1-pq+2\alpha pq-2\alpha+\alpha^2=x^2-\alpha(p+q)x-pq+2\alpha pq+(1-\alpha)^2>0$ for all $x>\lambda_\alpha(\overleftrightarrow{K_{p,q}})$.
Since $p>q$, $f(x)-g(x)>0$ for all $x>\lambda_\alpha(\overleftrightarrow{K_{p,q}})>1$. By Lemma \ref{le:9}, we have
$\lambda_\alpha(B^2_{n,p,q})>\lambda_\alpha(\overleftrightarrow{K_{p,q}})>1$. Hence
$f(x)-g(x)>0$ for all $x\geq\lambda_\alpha(B^2_{n,p,q})$. Then $\lambda_\alpha(B^2_{n,p,q})>\lambda_\alpha(B^1_{n,p,q})$.

Therefore, $ \lambda_\alpha(B^2_{n,p,q})\geq\lambda_\alpha(B^1_{n,p,q})$ with equality if and only if $p=q$.
\end{proof}

\noindent\begin{theorem}\label{th:ch-2} For digraphs $B^5_{n,p,q}$ and $B^6_{n,p,q}$, as shown in Figure \ref{Fig.2},
$$\lambda_\alpha(B^6_{n,p,q})\geq\lambda_\alpha(B^5_{n,p,q}),$$ with equality if and only if $p=q$ or $\alpha=0$.
\end{theorem}

\begin{proof} If $p=q$, then $B^5_{n,p,q}\cong B^6_{n,p,q}$. Hence $\lambda_\alpha(B^5_{n,p,q})=\lambda_\alpha(B^6_{n,p,q})$.
Otherwise $p>q$, let $G=B^5_{n,p,q}$. Similar to the proof Theorem \ref{th:ch-1}, we can know that $\lambda_\alpha(G)$ satisfy the follow equation
\begin{align*}
&\left(\frac{\lambda_\alpha(G)-\alpha}{1-\alpha}\right)^{n-p-q}[\lambda^3_\alpha(G)-(\alpha p+2\alpha q+\alpha)\lambda^2_\alpha(G)+(\alpha^2 q^2
+\alpha^2pq+2\alpha pq+\alpha^2q+\alpha^2p\\
&-pq)\lambda_\alpha(G)-2\alpha^2q^2p-2\alpha^2pq+\alpha q^2p+\alpha pq+2\alpha^2q-\alpha^3q-\alpha q]-(1-\alpha)^2 (\lambda_\alpha(G)-\alpha q)=0.
\end{align*}

Let $f(x)=\left(\frac{x-\alpha}{1-\alpha}\right)^{n-p-q}[x^3-(\alpha p+2\alpha q+\alpha)x^2+(\alpha^2 q^2+\alpha^2pq+2\alpha pq+\alpha^2q+\alpha^2p-pq)x
-2\alpha^2q^2p-2\alpha^2pq+\alpha q^2p+\alpha pq+2\alpha^2q-\alpha^3q-\alpha q]-(1-\alpha)^2 (x-\alpha q)$.
Then $\lambda_\alpha(B^5_{n,p,q})$ is the largest real root of $f(x)=0$. Similarly,
let $g(x)=\left(\frac{x-\alpha}{1-\alpha}\right)^{n-p-q}[x^3-(\alpha q+2\alpha p+\alpha)x^2+(\alpha^2 p^2+\alpha^2pq+2\alpha pq+\alpha^2p+\alpha^2q-pq)x
-2\alpha^2p^2q-2\alpha^2pq+\alpha p^2q+\alpha pq+2\alpha^2p-\alpha^3p-\alpha p]-(1-\alpha)^2 (x-\alpha p)$, then
$\lambda_\alpha(B^6_{n,p,q})$ is the largest real root of $g(x)=0$. Thus
\begin{align*}
f(x)-g(x)&=\left(\frac{x-\alpha}{1-\alpha}\right)^{n-p-q}\cdot\alpha\cdot(p-q)\cdot[x^2-\alpha(p+q)x-pq+2\alpha pq+(1-\alpha)^2]\\
&\ \ \ \ -(1-\alpha)^2\alpha(p-q).
\end{align*}
For $\alpha=0$, $f(x)=g(x)$, then $\lambda_\alpha(B^5_{n,p,q})=\lambda_\alpha(B^6_{n,p,q})$.
\\For $0<\alpha<1$.
Since $\lambda_\alpha(\overleftrightarrow{K_{p,q}})$ is the the largest real root of the equation $x^2-\alpha(p+q)x-pq+2\alpha pq=0$,
$x^2-\alpha(p+q)x-pq+2\alpha pq>0$ for all $x>\lambda_\alpha(\overleftrightarrow{K_{p,q}})$.
Thus $x^2-\alpha(p+q)x-pq+2\alpha pq+(1-\alpha)^2>(1-\alpha)^2$ for all $x>\lambda_\alpha(\overleftrightarrow{K_{p,q}})$.
Since $p>q$, $f(x)-g(x)>\left(\frac{x-\alpha}{1-\alpha}\right)^{n-p-q}(1-\alpha)^2\alpha(p-q)-(1-\alpha)^2\alpha(p-q)>0$ for all $x>\lambda_\alpha(\overleftrightarrow{K_{p,q}})>1$.
By Lemma \ref{le:9}, we have
$\lambda_\alpha(B^6_{n,p,q})>\lambda_\alpha(\overleftrightarrow{K_{p,q}})>1$. Hence
$f(x)-g(x)>0$ for all $x\geq\lambda_\alpha(B^6_{n,p,q})$. Then $\lambda_\alpha(B^6_{n,p,q})>\lambda_\alpha(B^5_{n,p,q})$.

Therefore, $\lambda_\alpha(B^6_{n,p,q})\geq\lambda_\alpha(B^5_{n,p,q})$ with equality if and only if $p=q$ or $\alpha=0$.
\end{proof}

\noindent\begin{theorem}\label{th:ch-3} Let $B^3_{n,p,q}=B^1_{n,p,q}-\{(v_n,v_p)\}+\{(v_n,v_1)\}$. Then
$$\lambda_\alpha(B^3_{n,p,q})>\lambda_\alpha(B^1_{n,p,q}).$$
\end{theorem}

\begin{proof}
Clearly $B^3_{n,p,q}$ is strongly connected. Let $X=(x_1,x_2,\ldots,x_n)^T$ be the
Perron vector corresponding to $\lambda_\alpha(B^1_{n,p,q})$, where $x_i$ corresponds to the vertex $v_i$. By Lemma
\ref{le:a}, we get $x_1>x_p$
Thus $\lambda_\alpha(B^3_{n,p,q})>\lambda_\alpha(B^1_{n,p,q})$ by Lemma \ref{le:4}.
\end{proof}

\noindent\begin{theorem}\label{th:ch-7} Let $B^4_{n,p,q}=B^2_{n,p,q}-\{(v_n,v_{p+q})\}+\{(v_n,v_{p+1})\}$. Then
$$\lambda_\alpha(B^4_{n,p,q})>\lambda_\alpha(B^2_{n,p,q}).$$
\end{theorem}

\begin{proof}
Clearly $B^4_{n,p,q}$ is strongly connected. Let $X=(x_1,x_2,\ldots,x_n)^T$ be the
Perron vector corresponding to $\lambda_\alpha(B^2_{n,p,q})$, where $x_i$ corresponds to the vertex $v_i$.
By Lemma \ref{le:a}, we get $x_{p+1}>x_{p+q}$.
Thus $\lambda_\alpha(B^4_{n,p,q})>\lambda_\alpha(B^2_{n,p,q})$ by Lemma \ref{le:4}.
\end{proof}

\noindent\begin{theorem}\label{th:ch-5} For digraphs $B^1_{n,p,q}$ and $B^5_{n,p,q}$, as shown in Figure \ref{Fig.1} and \ref{Fig.2},
$$\lambda_\alpha(B^5_{n-1,p,q})>\lambda_\alpha(B^1_{n,p,q}).$$
\end{theorem}

\begin{proof} Since $B^5_{n,p,q}=B^1_{n,p,q}-\{(v_n,v_p)\}+\{(v_n,v_{p+1})\}$ and $B^5_{n,p,q}$ is strongly connected. Let $X=(x_1,x_2,\ldots,x_n)^T$ be the
Perron vector corresponding to $\lambda_\alpha(B^1_{n,p,q})$, where $x_i$ corresponds to the vertex $v_i$.
In the following, we will prove $x_{p+1}>x_{p}$.

By Lemma \ref{le:a}, $x_2=x_3=\ldots=x_p\triangleq x_p$, $x_{p+1}=x_{p+2}=\ldots=x_{p+q}\triangleq x_{p+1}$ and $x_1>x_p$.
Therefore
\begin{align*}
\lambda_\alpha(B^1_{n,p,q})x_{p+1}&=\alpha px_{p+1}+(1-\alpha)x_1+(1-\alpha)(p-1)x_p\\
&>\alpha px_{p+1}+(1-\alpha)x_p+(1-\alpha)(p-1)x_p\\
&=\alpha p(x_{p+1}-x_p)+px_p,
\end{align*}
$$\lambda_\alpha(B^1_{n,p,q})x_{p}=\alpha qx_{p}+(1-\alpha)qx_{p+1}.$$
Hence
$$\lambda_\alpha(B^1_{n,p,q})(x_{p+1}-x_p)>(\alpha p+\alpha q-q)(x_{p+1}-x_p)+(p-q)x_p.$$
Furthermore
$$(\lambda_\alpha(B^1_{n,p,q}-(\alpha p+\alpha q-q))(x_{p+1}-x_p)>(p-q)x_p\geq0.$$
However, by Lemma \ref{le:3}, we get $\lambda_\alpha(B^1_{n,p,q})> \alpha\Delta^+\geq\alpha p>\alpha p+\alpha q-q$.
Thus $x_{p+1}>x_p$. Therefore $\lambda_\alpha(B^5_{n,p,q})>\lambda_\alpha(B^1_{n,p,q})$.
By Lemma \ref{le:11}, we have $\lambda_\alpha(B^5_{n-1,p,q})\geq\lambda_\alpha(B^5_{n,p,q})$.
Then $\lambda_\alpha(B^5_{n-1,p,q})>\lambda_\alpha(B^1_{n,p,q})$.
\end{proof}

\noindent\begin{theorem}\label{th:ch-8} For digraphs $B^1_{n,p,q}$ and $B^5_{n,p,q}$, as shown in Figure \ref{Fig.1} and \ref{Fig.2},
$$\lambda_\alpha(B^1_{n-1,p,q})\geq\lambda_\alpha(B^5_{n,p,q}).$$
\end{theorem}

\begin{proof} Let $B^{5*}_{n,p,q}=B^5_{n,p,q}-\{(v_{n-1},v_n)\}+\{(v_{n-1},v_{p})\}$.
Let $X=(x_1,x_2,\ldots,x_n)^T$ be the
Perron vector corresponding to $\lambda_\alpha(B^5_{n,p,q})$, where $x_i$ corresponds to the vertex $v_i$.
By Lemma \ref{le:a}, we get $x_p>x_n$. Then $\lambda_\alpha(B^{5*}_{n,p,q})\geq\lambda_\alpha(B^5_{n,p,q})$.
Since the indegree of $v_n$ is 0 in $B^{5*}_{n,p,q}$, $B^{5*}_{n,p,q}$ is not strongly connected
which contains $B^1_{n-1,p,q}$ as a induced subdigraph, we have $\lambda_\alpha(B^{5*}_{n,p,q})=\lambda_\alpha(B^1_{n-1,p,q})$.
Thus $\lambda_\alpha(B^1_{n-1,p,q})\geq\lambda_\alpha(B^5_{n,p,q})$.
\end{proof}

In the following, we give the main results of this section.

\noindent\begin{theorem}\label{th:ch-9} Let $p\geq q\geq2$, $p+q\leq n-1$, $n\equiv p+q \ (mod \ 2)$ and
$G\in\mathcal{G}_{n,p,q}$. Then

$(i)$ For $\alpha=0$, $\lambda_\alpha(G)\geq \lambda_\alpha(B^5_{n,p,q})=\lambda_\alpha(B^6_{n,p,q})$ and the equality holds
if and only if $G\cong B^5_{n,p,q}$ or $G\cong B^6_{n,p,q}$.

$(ii)$ For $0<\alpha<1$, $\lambda_\alpha(G)\geq \lambda_\alpha(B^5_{n,p,q})$ and the equality holds
if and only if $G\cong B^5_{n,p,q}$.
\end{theorem}

\begin{proof} Since $G\in\mathcal{G}_{n,p,q}$, $\overleftrightarrow{K_{p,q}}$ is a proper subdigraph of $G$.
Since $G$ is strongly connected, it is possible to obtain a digraph $H$ from $G$ by deleting vertices and arcs in
a way such that one has a subdigraph $\overleftrightarrow{K_{p,q}}$. Therefore

(1) $H\cong B^1_{p+q+k,p,q}$, \ \ \ \ ($k\equiv 1$ \ (mod \ 2), $k\geq1$) or

(2) $H\cong B^2_{p+q+k,p,q}$, \ \ \ \ ($k\equiv 1$ \ (mod \ 2), $k\geq1$) or

(3) $H\cong B^3_{p+q+k,p,q}$, \ \ \ \ ($k\equiv 1$ \ (mod \ 2), $k\geq1$) or

(4) $H\cong B^4_{p+q+k,p,q}$, \ \ \ \ ($k\equiv 1$ \ (mod \ 2), $k\geq1$) or

(5) $H\cong B^5_{p+q+l,p,q}$, \ \ \ \ \ ($l\equiv 0$ \ (mod \ 2), $l\geq2$) or

(6) $H\cong B^6_{p+q+l,p,q}$, \ \ \ \ \ ($l\equiv 0$ \ (mod \ 2), $l\geq2$).

By Lemma \ref{le:9}, $\lambda_\alpha(G)\geq\lambda_\alpha(H)$, the equality holds if and only if $H\cong G$.

{\bf Case (i).} $H\cong B^1_{p+q+k,p,q}$, \ \ \ \ ($k\equiv 1$ \ (mod \ 2), $k\geq1$)

Insert $n-p-q-k-1$ vertices into the directed path $P_{k+2}$ such that
the resulting bipartite digraph is $B^1_{n-1,p,q}$, then
$\lambda_\alpha(H)\geq\lambda_\alpha(B^1_{n-1,p,q})$ by using Lemma \ref{le:11} repeatedly $n-p-q-k-1$ times,
and thus $\lambda_\alpha(G)>\lambda_\alpha(H)\geq\lambda_\alpha(B^1_{n-1,p,q})\geq\lambda_\alpha(B^5_{n,p,q})$
by Theorem \ref{th:ch-8}.

{\bf Case (ii).} $H\cong B^2_{p+q+k,p,q}$, \ \ \ \ ($k\equiv 1$ \ (mod \ 2), $k\geq1$)

Insert $n-p-q-k-1$ vertices into the directed path $P_{k+2}$ such that
the resulting bipartite digraph is $B^2_{n-1,p,q}$, then
$\lambda_\alpha(H)\geq\lambda_\alpha(B^2_{n-1,p,q})$ by using Lemma \ref{le:11} repeatedly $n-p-q-k-1$ times,
and thus $\lambda_\alpha(G)>\lambda_\alpha(H)\geq\lambda_\alpha(B^2_{n-1,p,q})\geq\lambda_\alpha(B^1_{n-1,p,q})\geq\lambda_\alpha(B^5_{n,p,q})$
by Theorems \ref{th:ch-1} and \ref{th:ch-8}.

{\bf Case (iii).} $H\cong B^3_{p+q+k,p,q}$, \ \ \ \ ($k\equiv 1$ \ (mod \ 2), $k\geq1$)

Insert $n-p-q-k-1$ vertices into the directed cycle $C_{k+1}$ such that
the resulting bipartite digraph is $B^3_{n-1,p,q}$, then
$\lambda_\alpha(H)\geq\lambda_\alpha(B^3_{n-1,p,q})$ by using Lemma \ref{le:11} repeatedly $n-p-q-k-1$ times,
and thus $\lambda_\alpha(G)>\lambda_\alpha(H)\geq\lambda_\alpha(B^3_{n-1,p,q})>\lambda_\alpha(B^1_{n-1,p,q})\geq\lambda_\alpha(B^5_{n,p,q})$
by Theorems \ref{th:ch-3} and \ref{th:ch-8}.

{\bf Case (iv).} $H\cong B^4_{p+q+k,p,q}$, \ \ \ \ ($k\equiv 1$ \ (mod \ 2), $k\geq1$)

Insert $n-p-q-k-1$ vertices into the directed cycle $C_{k+1}$ such that
the resulting bipartite digraph is $B^4_{n-1,p,q}$, then
$\lambda_\alpha(H)\geq\lambda_\alpha(B^4_{n-1,p,q})$ by using Lemma \ref{le:11} repeatedly $n-p-q-k-1$ times,
and thus $\lambda_\alpha(G)>\lambda_\alpha(H)\geq\lambda_\alpha(B^4_{n-1,p,q})
>\lambda_\alpha(B^2_{n-1,p,q})\geq\lambda_\alpha(B^1_{n-1,p,q})\geq\lambda_\alpha(B^5_{n,p,q})$
by Theorems \ref{th:ch-7}, \ref{th:ch-1} and \ref{th:ch-8}.

{\bf Case (v).} $H\cong B^5_{p+q+l,p,q}$, \ \ \ \ ($l\equiv 0$ \ (mod \ 2), $l\geq2$)

Insert $n-p-q-l$ vertices into the directed path $P_{l+2}$ such that
the resulting bipartite digraph is $B^5_{n,p,q}$, then
$\lambda_\alpha(H)\geq\lambda_\alpha(B^5_{n,p,q})$ by using Lemma \ref{le:11} repeatedly $n-p-q-l$ times.
Hence, by Theorem \ref{th:ch-2}, we have

$(1)$ For $\alpha=0$, $\lambda_\alpha(G)\geq\lambda_\alpha(H)\geq \lambda_\alpha(B^5_{n,p,q})=\lambda_\alpha(B^6_{n,p,q})$.

$(2)$ For $0<\alpha<1$, $\lambda_\alpha(G)\geq\lambda_\alpha(H)\geq \lambda_\alpha(B^5_{n,p,q})$.

{\bf Case (vi).} $H\cong B^6_{p+q+l,p,q}$, \ \ \ \ ($l\equiv 0$ \ (mod \ 2), $l\geq2$)

Insert $n-p-q-l$ vertices into the directed path $P_{l+2}$ such that
the resulting bipartite digraph is $B^6_{n,p,q}$, then
$\lambda_\alpha(H)\geq\lambda_\alpha(B^6_{n,p,q})$ by using Lemma \ref{le:11} repeatedly $n-p-q-l$ times.
Hence, by Theorem \ref{th:ch-2}, we have

$(1)$ For $\alpha=0$, $\lambda_\alpha(G)\geq\lambda_\alpha(H)\geq \lambda_\alpha(B^5_{n,p,q})=\lambda_\alpha(B^6_{n,p,q})$.

$(2)$ For $0<\alpha<1$, $\lambda_\alpha(G)\geq\lambda_\alpha(H)\geq \lambda_\alpha(B^6_{n,p,q})\geq\lambda_\alpha(B^5_{n,p,q})$.

Combining the above six cases, we have

$(1)$ For $\alpha=0$, $\lambda_\alpha(G)\geq \lambda_\alpha(B^5_{n,p,q})=\lambda_\alpha(B^6_{n,p,q})$ and the equality holds
if and only if $G\cong B^5_{n,p,q}$ or $G\cong B^6_{n,p,q}$.

$(2)$ For $0<\alpha<1$, $\lambda_\alpha(G)\geq \lambda_\alpha(B^5_{n,p,q})$ and the equality holds
if and only if $G\cong B^5_{n,p,q}$.
\end{proof}

\noindent\begin{theorem}\label{th:ch-10} Let $p\geq q\geq2$, $p+q\leq n-1$, $n\not\equiv p+q \ (mod \ 2)$ and
$G\in\mathcal{G}_{n,p,q}$. Then $\lambda_\alpha(G)\geq \lambda_\alpha(B^1_{n,p,q})$ and the equality holds
if and only if $G\cong B^1_{n,p,q}$.
\end{theorem}

\begin{proof} Since $G\in\mathcal{G}_{n,p,q}$, $\overleftrightarrow{K_{p,q}}$ is a proper subdigraph of $G$.
Since $G$ is strongly connected, it is possible to obtain a digraph $H$ from $G$ by deleting vertices and arcs in
a way such that one has a subdigraph $\overleftrightarrow{K_{p,q}}$. Therefore

(1) $H\cong B^1_{p+q+k,p,q}$, \ \ \ \ ($k\equiv 1$ \ (mod \ 2), $k\geq1$) or

(2) $H\cong B^2_{p+q+k,p,q}$, \ \ \ \ ($k\equiv 1$ \ (mod \ 2), $k\geq1$) or

(3) $H\cong B^3_{p+q+k,p,q}$, \ \ \ \ ($k\equiv 1$ \ (mod \ 2), $k\geq1$) or

(4) $H\cong B^4_{p+q+k,p,q}$, \ \ \ \ ($k\equiv 1$ \ (mod \ 2), $k\geq1$) or

(5) $H\cong B^5_{p+q+l,p,q}$, \ \ \ \ ($l\equiv 0$ \ (mod \ 2), $l\geq2$) or

(6) $H\cong B^6_{p+q+l,p,q}$, \ \ \ \ ($l\equiv 0$ \ (mod \ 2), $l\geq2$).

By Lemma \ref{le:9}, $\lambda_\alpha(G)\geq\lambda_\alpha(H)$, the equality holds if and only if $H\cong G$.

{\bf Case (i).} $H\cong B^1_{p+q+k,p,q}$, \ \ \ \ ($k\equiv 1$ \ (mod \ 2), $k\geq1$)

Insert $n-p-q-k$ vertices into the directed path $P_{k+2}$ such that
the resulting bipartite digraph is $B^1_{n,p,q}$, then
$\lambda_\alpha(H)\geq\lambda_\alpha(B^1_{n,p,q})$ by using Lemma \ref{le:11} repeatedly $n-p-q-k$ times,
and thus $\lambda_\alpha(G)\geq\lambda_\alpha(H)\geq\lambda_\alpha(B^1_{n,p,q})$

{\bf Case (ii).} $H\cong B^2_{p+q+k,p,q}$, \ \ \ \ ($k\equiv 1$ \ (mod \ 2), $k\geq1$)

Insert $n-p-q-k$ vertices into the directed path $P_{k+2}$ such that
the resulting bipartite digraph is $B^2_{n,p,q}$, then
$\lambda_\alpha(H)\geq\lambda_\alpha(B^2_{n,p,q})$ by using Lemma \ref{le:11} repeatedly $n-p-q-k$ times,
and thus $\lambda_\alpha(G)\geq\lambda_\alpha(H)\geq\lambda_\alpha(B^2_{n,p,q})\geq\lambda_\alpha(B^1_{n,p,q})$
by Theorem \ref{th:ch-1}.

{\bf Case (iii).} $H\cong B^3_{p+q+k,p,q}$, \ \ \ \ ($k\equiv 1$ \ (mod \ 2), $k\geq1$)

Insert $n-p-q-k$ vertices into the directed cycle $C_{k+1}$ such that
the resulting bipartite digraph is $B^3_{n,p,q}$, then
$\lambda_\alpha(H)\geq\lambda_\alpha(B^3_{n,p,q})$ by using Lemma \ref{le:11} repeatedly $n-p-q-k$ times,
and thus $\lambda_\alpha(G)\geq\lambda_\alpha(H)\geq\lambda_\alpha(B^3_{n,p,q})>\lambda_\alpha(B^1_{n,p,q})$
by Theorem \ref{th:ch-3}.

{\bf Case (iv).} $H\cong B^4_{p+q+k,p,q}$, \ \ \ \ ($k\equiv 1$ \ (mod \ 2), $k\geq1$)

Insert $n-p-q-k$ vertices into the directed cycle $C_{k+1}$ such that
the resulting bipartite digraph is $B^4_{n,p,q}$, then
$\lambda_\alpha(H)\geq\lambda_\alpha(B^4_{n,p,q})$ by using Lemma \ref{le:11} repeatedly $n-p-q-k$ times,
and thus $\lambda_\alpha(G)\geq\lambda_\alpha(H)\geq\lambda_\alpha(B^4_{n,p,q})
>\lambda_\alpha(B^2_{n,p,q})\geq\lambda_\alpha(B^1_{n,p,q})$.
by Theorems \ref{th:ch-7} and \ref{th:ch-1}.

{\bf Case (v).} $H\cong B^5_{p+q+l,p,q}$, \ \ \ \ ($l\equiv 0$ \ (mod \ 2), $l\geq2$)

Insert $n-p-q-l-1$ vertices into the directed path $P_{l+2}$ such that
the resulting bipartite digraph is $B^5_{n-1,p,q}$, then
$\lambda_\alpha(H)\geq\lambda_\alpha(B^5_{n-1,p,q})$ by using Lemma \ref{le:11} repeatedly $n-p-q-l-1$ times.
Hence, by Theorem \ref{th:ch-5}, we have
 $\lambda_\alpha(G)>\lambda_\alpha(H)\geq \lambda_\alpha(B^5_{n-1,p,q})>\lambda_\alpha(B^1_{n,p,q})$.

{\bf Case (vi).} $H\cong B^6_{p+q+l,p,q}$, \ \ \ \ ($l\equiv 0$ \ (mod \ 2), $l\geq2$)

Insert $n-p-q-l-1$ vertices into the directed path $P_{l+2}$ such that
the resulting bipartite digraph is $B^6_{n-1,p,q}$, then
$\lambda_\alpha(H)\geq\lambda_\alpha(B^6_{n-1,p,q})$ by using Lemma \ref{le:11} repeatedly $n-p-q-l-1$ times.
Hence, by Theorems \ref{th:ch-2} and \ref{th:ch-5}, we have
$\lambda_\alpha(G)>\lambda_\alpha(H)\geq \lambda_\alpha(B^6_{n-1,p,q})\geq\lambda_\alpha(B^5_{n-1,p,q})>\lambda_\alpha(B^1_{n,p,q})$.

Combining the above six cases, we have $\lambda_\alpha(G)\geq \lambda_\alpha(B^1_{n,p,q})$ and the equality holds
if and only if $G\cong B^1_{n,p,q}$.
\end{proof}


\begin{thebibliography}{99}
\bibitem{BP} A. Berman, R. J. Plemmons, Nonnegative Matrices in the Mathematical Sciences, New York: Academic Press, 1979.

\bibitem{CCL} S.T. Chen, S.L. Chen, W.Q. Liu, The minimum spectral radius of strongly connected
bipartite digraphs with complete bipartite subdigraph, Quantitative Logic and Soft Computing 2016,
Springer International Publishing 2017, 659-669.

\bibitem{GM} H.A. Ganie, M. Baghipur, on the generalized adjacency spectral radius of digraphs, Linear Multilinear Algebra,
https://doi.org/10.1080/03081087.2020.1844614.

\bibitem{GL} G.Q. Guo, J. Liu, Some results on the spectral radius of generalized $\infty$ and $\theta$-digraphs,
Linear Algebra Appl., 437 (2012), 2200-2208.

\bibitem{HJ} R.A. Horn, C.R. Johnson, Matrix Analysis, Cambridge University
Press, New York, 1985.

\bibitem{HoYo} W.X. Hong, L.H. You, Spectral radius and signless Laplacian spectral
radius of strongly connected digraphs, Linear Algebra Appl., 457 (2014), 93-113.

\bibitem{LZ} J. Li, B. Zhou, On the spectral radius of strongly connected digraphs, Bull Iranian Math. Soc., 41 (2015), 381-387.

\bibitem{LD} H.Q. Lin, S.W. Drury, The maximum perron roots
of digraphs with some given parameters, Discrete Math., 313 (2013), 2607-2613.

\bibitem{LWZ} X.H. Li, L.G. Wang  S.Y Zhang, The signless Laplacian spectral radius of some
strongly connected digraphs, Indian J. Pure Appl. Math., 49 (2018), 113-127.

\bibitem{LS1} H.Q. Lin, J.L. Shu, Spectral radius of digraphs with
given dichromatic number, Linear Algebra  Appl., 434 (2011), 2462-2467.

\bibitem{LS} H.Q. Lin, J.L. Shu, A note on the spectral characterization of strongly
connected bicyclic digraphs, Linear Algebra Appl., 436 (2012), 2524-2530.

\bibitem{LSWY} H.Q. Lin, J.L. Shu, Y.R. Wu, G.L. Yu,
Spectral radius of strongly connected digraphs, Discrete Math.,
312 (2012), 3663-3669.

\bibitem{LiLi} H.Q. Lin, X.G. Liu, J. Xue, Graphs determined by their $A_\alpha$-spectra, Discrete Math., 342, (2019), 441-450.

\bibitem{LWCL} J.P. Liu, X.Z. W, J.S. Chen, B.L. Liu, The $A_\alpha$ spectral radius characterization of some digraphs, Linear Algebra Appl.,
563 (2019), 63-74.

\bibitem{LL} X.G. Liu, S.Y. Liu, On the $A_\alpha$-characteristic polynomial of a graph, Linear Algebra Appl., 546 (2018), 274-288.

\bibitem{Niki} V. Nikiforov, Merging the $A$- and $Q$-spectral theories, Applicable Analysis and Discrete Math., 11 (2017), 81-107.

\bibitem{NR} V. Nikiforov, O. Rojo, On the $\alpha$-index of graphs with pendent paths, Linear
Algebra Appl., 550 (2018), 87-104.

\bibitem{NO} V. Nikiforov, O. Rojo, A note on the positive semidefiniteness of $A_\alpha(G)$, Linear Algebra Appl., 519 (2017), 156-163.

\bibitem{XiWa1} W.G. Xi, L.G. Wang, Sharp upper bounds on the signless laplacian
spectral radius of strongly connected digraphs, Discuss. Math. Graph Theory, 36 (2016), 977-988.

\bibitem{XiWa2} W.G. Xi, L.G. Wang, The signless Laplacian and distance signless Laplacian
spectral radius of digraphs with some given parameters, Discrete Appl. Math., 227 (2017), 136-141.

\bibitem{XiWa4} W.G. Xi, W. So, L.G. Wang, On the $A_\alpha$ spectral radius of digraphs with given parameters, Linear Multilinear Algebra, https://doi.org/10.1080/03081087.2020.1793879.

\bibitem{XLLS} J. Xue, H.Q. Lin, S.T. Liu, J.L. Shu, On the $A_\alpha$-spectral radius of a graph, Linear Algebra Appl., 550 (2018), 105-120.
\end{thebibliography}
\end{document}